\documentclass[reqno, 11pt]{amsart}
\usepackage{amsrefs}
\usepackage{graphics}
\usepackage{amssymb, amsmath}
\usepackage{mathtools}   
\usepackage{fancybox}
\usepackage{bm}
\usepackage{mathrsfs}  %参考文献の部分
\usepackage{esint}
\usepackage{enumitem}
\usepackage{comment} %変更をやる時に使う change を使いたい場合は使うな. 

\usepackage{cite} %%これだけで, 番号順に参考文献ならぶ

\makeatletter

\@addtoreset{equation}{section}
\makeatother

\DeclareMathOperator*{\esssup}{ess \, sup}

\allowdisplaybreaks[4]

\usepackage{aliascnt}

\newtheorem{theorem}{\sc Theorem}[section]

\newaliascnt{lemma}{theorem}
\newtheorem{lemma}[lemma]{\sc Lemma}
\aliascntresetthe{lemma}

\newaliascnt{proposition}{theorem}
\newtheorem{proposition}[proposition]{\sc Proposition}
\aliascntresetthe{proposition}

\newaliascnt{corollary}{theorem}

\aliascntresetthe{corollary}

\newaliascnt{definition}{theorem}
\newtheorem{definition}[definition]{\sc Definition}
\aliascntresetthe{definition}

\newaliascnt{example}{theorem}

\aliascntresetthe{example}

\theoremstyle{remark}
\newaliascnt{remark}{theorem}
\newtheorem{remark}[remark]{\sc Remark}
\aliascntresetthe{remark}

\usepackage[nameinlink,capitalize,noabbrev]{cleveref}

\crefname{theorem}{Theorem}{Theorems}
\Crefname{theorem}{Theorem}{Theorems}

\crefname{lemma}{Lemma}{Lemmas}
\Crefname{lemma}{Lemma}{Lemmas}

\crefname{proposition}{Proposition}{Propositions}
\Crefname{proposition}{Proposition}{Propositions}

\crefname{corollary}{Corollary}{Corollaries}
\Crefname{corollary}{Corollary}{Corollaries}

\crefname{definition}{Definition}{Definitions}
\Crefname{definition}{Definition}{Definitions}

\crefname{example}{Example}{Examples}
\Crefname{example}{Example}{Examples}

\crefname{remark}{Remark}{Remarks}
\Crefname{remark}{Remark}{Remarks}

\crefname{section}{Section}{Sections}  % もともとのまま
\crefname{appendix}{Appendix}{Appendices}  % appendix 用の型を追加

\usepackage{xcolor}        % 色指定用（必須ではないが推奨）
\makeatletter
\newcommand{\printaddresses}{%
  \@setaddresses
  \global\let\@setaddresses\relax
}
\makeatother

\begin{document}

%\shorttitle{Time-fractional nonlinear evolution equations} 

\title[Time-fractional evolution equations]{Time-fractional nonlinear evolution equations with time-dependent constraints}

\author{Yoshihito Nakajima}
\thanks{ORCID: 0009-0009-7152-9589} %https://orcid.org/0009-0009-7152-9589

 \address[Yoshihito Nakajima]{Graduate School of Science, Tohoku University, Aoba, Sendai 980-8578, Japan}
 \email{naka.donda@gmail.com}

\date{\today}


\maketitle
\printaddresses 

\begin{abstract}
This article is devoted to developing an abstract theory of
time-fractional gradient flow equations for time-dependent convex
functionals in real Hilbert spaces. The main results concern the
existence of strong solutions to time-fractional abstract evolution
equations governed by subdifferential operators of time-dependent
convex functionals. In the classical theory of gradient flow
equations, chain-rule formulae play a crucial role in various
analyses, and such formulae for subdifferentials of time-dependent
functionals are also known in the case of first-order time
derivatives. In contrast, in the present setting, the presence of
time-fractional derivatives prevents the direct use of the usual
chain-rule. To overcome this difficulty, fractional chain-rule
formulae for subdifferentials of time-dependent convex functionals are
established under a nonlocal variant of the so-called Kenmochi 
condition. Moreover, Gronwall-type lemmas for nonlinear Volterra
integral inequalities are developed. Finally, the abstract results
obtained are applied to initial-boundary value problems for
time-fractional degenerate parabolic equations on moving domains.
\end{abstract}

\begin{comment} 
This article is devoted to presenting an abstract theory of time-fractional gradient flow equations for time-dependent convex functionals in real Hilbert spaces.
The main results are concerned with the existence of strong solutions to time-fractional abstract evolution equations governed by time-dependent subdifferential operators.
In the present setting, time-fractional abstract evolution equations are governed by time-dependent subdifferential operators, and the corresponding energy functionals explicitly depend on time.
This framework is not covered by existing abstract theories for time-fractional gradient flow equations, which are mainly formulated for time-independent energy functionals.
The presence of time-fractional derivatives causes fundamental analytical difficulties.
In the classical theory of gradient flow equations, chain-rule formulae play a central role in the analysis. 
In contrast, for time-fractional derivatives, the usual chain-rule formula is no longer available.
To overcome this difficulty, fractional chain-rule formulae for time-dependent subdifferential operators are established, and Gronwall-type lemmas for nonlinear Volterra integral inequalities are developed.
Moreover, the obtained abstract results are applied to the initial-boundary value problem for time-fractional nonlinear parabolic equations on moving domains. 
重複が多めであった. 
\end{comment}

%å\subjclass[2020]{Primary: 47J35; Secondary: 35K61}

\section{Introduction}

Throughout this paper, let $T \in (0,\infty)$ be fixed, and let $H$ be a real Hilbert space equipped with an inner product $( \bullet ,  \bullet )_{H}$ and the norm $\| \bullet  \|_{H} = \sqrt{( \bullet ,  \bullet )_{H}}$. 
For each $t \in [0,T]$, let $\varphi^{t} \colon H \to (-\infty,\infty]$ be a proper (i.e., $\varphi^{t} \not\equiv \infty$) lower-semicontinuous convex functional with the \emph{effective domain}, 
\begin{equation*}
D(\varphi^{t}) := \{ w \in H \colon \varphi^{t}(w) < \infty \} \neq \emptyset.
\end{equation*}
The subdifferential operator $\partial \varphi^{t} \colon H \to 2^{H}$ is defined by 
\begin{equation*}
\partial \varphi^{t}(z) := \{ \xi \in H \colon \varphi^{t}(v) - \varphi^{t}(z) \ge (\xi, v-z)_{H} \textup{ for all } v \in D(\varphi^{t}) \} 
\end{equation*}
for $z \in H$, and its \emph{domain} is denoted by 
\begin{equation*} 
D(\partial \varphi^{t}) := \{ w \in D(\varphi^{t}) \colon \partial \varphi^{t}(w) \neq \emptyset \}.
\end{equation*} 

\noindent For $k \in L^{1}(0,T)$ and $w \in L^{1}(0,T;H)$, the convolution $k*w$ is given by 
\begin{equation*}
  (k*w)(t) := \int_{0}^{t} k(t-s) w(s) \, \textup{d}s 
  \quad \textup{for } t \in (0,T).
\end{equation*}

In this article, we consider the following abstract Cauchy problem:
\begin{equation}  \tag{P} \label{E:main-equation1}
\partial_{t}[k * (u-u_{0})](t) 
+ \partial \varphi^{t}\bigl(u(t)\bigr) \ni f(t)
\quad \textup{in } H \quad \textup{for $t \in (0,T)$,}
\end{equation}
where $u_{0} \in H$ and $f \colon (0,T)\to H$ are prescribed,  and $k$ is a kernel satisfying the following condition:    

\begin{itemize}
\item[\textup{(PC)}] 
The kernel $k \in L^{1}_{\textup{loc}}([0,\infty))$ is nonnegative and nonincreasing. 
There exists a nonnegative and nonincreasing kernel $\ell \in L^{1}_{\textup{loc}}([0,\infty))$ such that 
\begin{equation*}
(k * \ell)(t) = \int_{0}^{t} k(t-s) \ell(s) \, \textup{d}s = 1 
\quad \textup{for all } t \in (0,\infty).
\end{equation*}
\end{itemize}
Therefore $k$ is a completely positive kernel (see \cite[Theorem 2.2]{A-Clement-1981}). 
Under this assumption, we write $(k, \ell) \in PC$. 

A typical example satisfying (PC) is the \emph{Riemann--Liouville kernel}, 
\begin{equation*}
k_{\alpha}(t) = \frac{t^{\alpha-1}}{\Gamma(\alpha)} 
\quad \textup{for $t \in (0,\infty)$  and $\alpha \in (0,1)$. } 
\end{equation*}
It is known that $(k_{1-\alpha}, k_{\alpha}) \in PC$ for $\alpha \in (0,1)$.   
In this case, $\partial_{t}[k_{1-\alpha} * (u - u_{0})](t)$ coincides with the $\alpha$-th order Riemann--Liouville fractional derivative  of $u - u_{0}$, and it agrees with the $\alpha$-th order Caputo derivative of $u$, provided that $u$ is a sufficiently smooth function satisfying $u(0) = u_{0}$.  
%Further examples of kernels satisfying \textup{(PC)} can be found in \cite[Section 6]{A-VergaraZacher-2015}.

The study of classical gradient flows dates back to the early work of Ha\"im Br\'ezis (see, e.g.,~\cite{B-Brezis-1973}) on the evolution equation, 
\begin{equation*}
\partial_{t} u(t) + \partial \varphi \bigl( u(t) \bigr) \ni f(t)
\quad \textup{in } H \quad \textup{for $t \in (0,T)$,} \quad u(0) = u_{0},  
\end{equation*}
for a proper lower-semicontinuous convex functional $\varphi$ on a real Hilbert space $H$. 
This fundamental framework is nowadays referred to as the \emph{Br\'ezis--K\=omura theory} (see also~\cite{A-Komura-1967}). 
Building upon this theory, Kenmochi~\cite{A-Kenmochi-1975} subsequently extended the Br\'ezis--K\=omura theory to cover evolution equations governed by \emph{time-dependent} subdifferential operators,
\begin{equation*}
\partial_{t}u(t) + \partial \varphi^{t} \bigl( u(t) \bigr) \ni f(t)
\quad \textup{in } H \quad \textup{for $t \in (0,T)$,} \quad u(0) = u_{0}.
\end{equation*}
This generalization makes it possible to treat nonlinear parabolic equations on moving domains and pave the way for applications to free boundary problems such as Stefan problems (see \cite{A-Kenmochi-1981}). 
Stefan problems typically arise in solid-liquid phase transitions, where heat conduction causes melting or solidification, and the interface between phases evolves over time.
As a consequence, the corresponding energy functionals explicitly depend on time, and hence, such a configuration is beyond the scope of the Br\'ezis--K\=omura theory. 
We also refer the reader to~\cite{A-Otani-1993/94,A-Watanabe-1973,A-Yamada-1976}  for related results and further developments on abstract evolution equations governed by time-dependent subdifferential operators. 

In recent years, there have been several extensions of these classical results on gradient flows to \emph{time-fractional} variants.  
In~\cite{A-Akagi-2019}, the well-posedness of the abstract Cauchy problem, 
\begin{equation}  %\tag{$\ast$} 
\label{Subdiff.Time-Fractional Equation}
\partial_{t}[k * (u - u_0)](t) + \partial \varphi \bigl( u(t) \bigr) \ni f(t)
\quad \textup{in } H \quad \textup{for $t \in (0,T)$,} 
\end{equation}
was established for proper lower-semicontinuous convex functionals $\varphi \colon H \to (-\infty,\infty]$.  
Moreover, it is also extended to Lipscitz perturbation problems. Furthermore, the abstract theory has been applied to time-fractional variants of nonlinear diffusion equations as well as Allen--Cahn equations. 
The time-fractional nonlinear diffusion equations are also studied in \cite{A-WittboldWolejkoZacher-2021,A-SchmitzWittbold-2024,A-BonforteGualdaniIbarrondo-2026,P-AkagiSodiniStefanelli-2025}. See also \cite{A-LiSalgado-2023} for a time-incremental construction of strong solutions to \eqref{Subdiff.Time-Fractional Equation}. 
Moreover, \cite{P-AkagiNakajima-2025} presents an abstract theory for  time-fractional evolution equations 
governed by the difference of two subdifferential operators of the form, 
\begin{equation*} 
  \partial_{t}[k * (u - u_0)](t)
  + \partial \varphi^{1} \bigl( u(t) \bigr)
  - \partial \varphi^{2} \bigl( u(t) \bigr)
  \ni f(t)
  \quad \textup{in } H \quad \textup{for $t \in (0,T)$}
\end{equation*}
(cf.~\cite{A-Otani-1977}). 
For related results, we refer the reader to \cite{B-AchleitnerAkagiKuehnMelenkRademacherSoresinaYang-2024,B-GalWarma-2020,B-KubicaRyszewskaYamamoto-2020,A-VergaraZacher-2015,A-VergaraZacher-2017,A-VergaraZacher-2008,A-Zacher-2008,A-Zacher-2009-abstract,B-Zhou-2024} and references therein. 
Among these results, we highlight that they rely on the pioneering work \cite{A-VergaraZacher-2008} (see also \cite{A-Zacher-2009-abstract}), which laid the groundwork for combining nonlocal time-differential operators with fractional calculus in the analysis of time-fractional evolution equations. 
However, the theory of \emph{time-fractional evolution equations} governed by \emph{time-dependent subdifferential operators} has not yet been well developed so far. 

The main purpose of the present paper is to establish an abstract theory concerning the existence of strong solutions to the Cauchy problem~\eqref{E:main-equation1}. 
To this end, several significant difficulties arise from the presence of time-fractional derivatives. 
We immediately face a major difficulty due to the lack of a valid chain-rule formula for time-fractional derivatives: 
while the chain-rule formula is a crucial tool in the study of gradient flow equations, the usual chain-rule formula is no longer valid for time-fractional derivatives (see,~e.g.,~\cite{A-Tarasov-2016}).  
Nevertheless, an alternative formula was established in~\cite{A-Akagi-2019}, which yields a fractional chain-rule formula for subdifferential operators in a practical form instead of the usual identities. 
However, this formula is restricted to time-\emph{independent} subdifferential operators, and hence, in order to handle the Cauchy problem \eqref{E:main-equation1}, where the subdifferential operator explicitly depends on time, we shall develop a time-dependent version of the fractional chain-rule formula. 
Moreover, we also develop Gronwall-type lemmas for nonlinear Volterra integral inequalities.

\bigskip
\noindent
{\bf Plan of the paper.}  
This paper is composed of eight sections.  
In \cref{Sec: Main results}, we present the main results of this paper, where we establish an existence result for the abstract Cauchy problem~\eqref{E:main-equation1}. 

In \cref{Sec: Preliminaries}, we collect preliminary facts
used throughout this paper. We also recall basic properties of subdifferential operators (see~\cref{subsec:subdifferential}), 
time-dependent subdifferential operators
(see~\cref{subsec:Time-dependent subdifferential}),
as well as nonlocal time-differential operators related to time-fractional derivatives
(see~\cref{subsec:time-nonlocal-ops}). 
In \cref{Sec: Some devices}, we develop Gronwall-type lemmas for nonlinear Volterra integral inequalities and establish time-fractional chain-rule formulae for time-dependent subdifferential operators, which plays a key role in the proof of our main results.   
\cref{Sec: Proof of well-posedness,Sec: proof of existence of strong solutions for Sobolev spaces,Sec: proof of existence of strong solutions for Lebesgue spaces} are devoted to giving proofs for our main results. 
More precisely, in \cref{Sec: Proof of well-posedness}, we prove the uniqueness and continuous dependence on initial data of strong solutions to \eqref{E:main-equation1}.  
In \cref{Sec: proof of existence of strong solutions for Sobolev spaces}, we prove the existence of strong solutions to \eqref{E:main-equation1} for $f \in W^{1,2}(0,T;H)$ under a certain assumption. 
In \cref{Sec: proof of existence of strong solutions for Lebesgue spaces}, this result is extended to $f \in L^{2}(0,T;H)$ under an additional assumption.  
\cref{Sec: Application} is dedicated to applications of the abstract theory to the Cauchy--Dirichlet problem for  time-fractional $p$-Laplace subdiffusion equations on moving domains.  
%In \cref{Sec: Appendix}, we provide a proof of auxiliary results on time-dependent subdifferential operators that are used in the proofs of the main theorems.

%For each $1 < r < \infty$, we denote by $r'$ the H\^{o}lder conjugate of $r$, that is, $1/r+1/r'=1$.  

%We denote by $C$ a generic nonnegative constant which may change from line to line. 

\bigskip
\noindent
{\bf Notation.}  
We use the same symbol $I$ for identity mappings on any space whenever no confusion arises.  
For a set-valued operator $A \colon H \to 2^{H}$, the domain $D(A)$, the range $R(A)$, and the graph $G(A)$ are defined by  
\begin{gather*} 
D(A) := \{ x \in H : A x \neq \emptyset \}, 
\quad 
R(A) := \bigcup_{x \in D(A)} A x, \\
G(A) := \{ (x,y) \in H \times H : y \in A x \}, 
\end{gather*} 
respectively. 
If the set $A z$ is a singleton for every $z \in D(A)$,  
we identify $A \colon H \to 2^{H}$ with a single-valued operator $A \colon D(A) \subset H \to H$;  
in this case, $A z$ denotes the unique element of the set $A z$.  
The inverse $A^{-1} \colon H \to 2^{H}$ is defined by  
\begin{equation*}
  A^{-1}x := \{ z \in H : x \in A z \} \quad \textup{for $x \in H$.}
\end{equation*}

%%%%%%%%%%%%%%%%%%%%%%%%%%%
%%%%%%%%%%%%%%%%%%%%%%%%%%%
%%%%%%%%%%%%%%%%%%%%%%%%%%%
%%%%%%%%%%%%%%%%%%%%%%%%%%%
%%%%%%%%%%%%%%%%%%%%%%%%%%%
%%%%%%%%%%%%%%%%%%%%%%%%%%%
%%%%%%%%%%%%%%%%%%%%%%%%%%%
%%%%%%%%%%%%%%%%%%%%%%%%%%%

\section{Main results} \label{Sec: Main results}

In this section, we present main results concerning existence of strong solutions to the abstract Cauchy problem~\eqref{E:main-equation1}. 
For each $t \in [0,T]$, let $\varphi^{t} \colon H \to (-\infty,\infty]$ be a proper (i.e., $\varphi^{t} \not\equiv \infty$) lower-semicontinuous convex functional. 
We define a notion of strong solution for the Cauchy problem~\eqref{E:main-equation1} in the following sense:  

\begin{definition}[Strong solutions to \eqref{E:main-equation1}]  
A function $u \in L^{2}(0,T;H)$ is called a \emph{strong solution} on $[0,T]$ to \eqref{E:main-equation1}  for  $(u_{0},f) \in H \times L^{2}(0,T;H)$,  if the following conditions are satisfied\/\textup{:} 
  \begin{itemize}
    \item[\textup{(i)}]
    It holds that $k*(u-u_{0}) \in W^{1,2}(0,T;H)$, $[k*(u-u_{0})](0) = 0$, and $u(t) \in D(\partial \varphi^{t})$ for~a.e.~$t \in (0,T)$. 
    \item[\textup{(ii)}]
    There exists $\xi \in L^{2}(0,T;H)$ such that 
    \[
      \xi(t) \in \partial \varphi^{t} \bigl( u(t) \bigr), \quad
      \partial_{t}[k*(u-u_{0})](t) + \xi(t) = f(t)
    \]
  for~a.e.~$t \in (0,T)$.
  \end{itemize}
\end{definition}
%From the definition of strong solutions and a fractional chain-rule inequality, 
We write $(u,\xi) \in \eqref{E:main-equation1}_{u_{0}, \, f}$ if $u,\xi \in L^{2}(0,T;H)$ satisfy the above conditions. 
We are ready for stating main results of this paper. 
The following theorem concerns the uniqueness and continuous dependence on initial data of strong solutions to \eqref{E:main-equation1}.

\begin{theorem}[Uniqueness and continuous dependence on initial data  of strong solutions to \eqref{E:main-equation1}]\label{thm:wellposedness}
Let $(k,\ell) \in PC$, and let  $(u_{i},\xi_{i}) \in \eqref{E:main-equation1}_{u_{0,i},\,f_{i}}$ for $i=1,2$,  where $u_{0,i} \in H$ and $f_{i} \in L^{2}(0,T;H)$. 
Then there exists a constant  $c_{T} \in [0,\infty)$  independent of $u_{0,1}$, $u_{0,2}$, $f_{1}$ and $f_{2}$ such that  
\begin{align*} 
  \|u_{1}-u_{2}\|_{L^{2}(0,T;H)}^{2} 
  \leq c_{T} \left( \|u_{0,1}-u_{0,2}\|_{H}^{2} + \|f_{1}-f_{2}\|_{L^{2}(0,T;H)}^{2} \right).
\end{align*}
In particular, if $u_{0,1}=u_{0,2}$ and $f_{1}=f_{2}$, then $u_{1}=u_{2}$. 
\end{theorem}

In order to state existence results, we employ the so-called Kenmochi condition   (cf. \cite{A-Kenmochi-1975,A-Kenmochi-1981,A-Yamada-1976,A-Otani-1993/94}): 

\begin{itemize}  
  \item[\textup{(A1)}]
  There is a constant $c_{1} \in [0,\infty)$ with the following property: for each $s,t \in [0,T]$ with $s \leq t$ and each $z_{s} \in D(\varphi^{s})$, there exists $z_{s,t} \in D(\varphi^{t})$ such that 
    \begin{align}
      \|z_{s,t} - z_{s}\|_{H} &\leq c_{1} |t - s| \bigl( 1+|\varphi^{s}(z_{s})| \bigr)^{1/2}, 
      \label{eq:A1-z-approx1}
      \\
      \varphi^{t}(z_{s,t})    &\leq \varphi^{s}(z_{s}) + c_{1} |t-s| \bigl( 1+|\varphi^{s}(z_{s})| \bigr). 
      \label{eq:A1-z-approx2} 
    \end{align}
\end{itemize}

We are now ready to state an existence result. 
The following theorem is concerned with the existence of strong solutions to \eqref{E:main-equation1}  for $f \in W^{1,2}(0,T;H)$: 

\begin{theorem}[Existence of strong solutions for $f \in W^{1,2}(0,T;H)$] \label{thm:main1}
Let $(k,\ell) \in PC$ and assume that \textup{(A1)} holds. 
Then, for every $u_{0} \in D(\varphi^{0})$ and $f \in W^{1,2}(0,T;H)$, the Cauchy problem \eqref{E:main-equation1} admits a unique strong solution $u \in L^{2}(0,T;H)$ on $[0,T]$  such that   
\begin{gather} \label{eq:energy-estimate-strong-Sobolev}
  \varphi^{ \bullet }  \bigl( u( \bullet ) \bigr) \in L^{\infty}(0,T),
  \quad 
  \ell* \|\partial_{t}[k*(u-u_{0})]\|_{H}^{2} \in L^{\infty}(0,T). 
\end{gather} 
 Furthermore, it holds that $u$ belongs to $C([0,T];H)$ and $u(0) = u_{0}$. 
\end{theorem}

The above existence result can be extended to  more general external forces under the following additional assumption:  

\begin{itemize}  
  \item[ \textup{(A2)} ]  
  There exists a constant $c_{2} \in [0,\infty)$ such that the following property holds: for each $u \in L^{2}(0,T;H)$  satisfying  $\varphi^{ \bullet }(u( \bullet )) \in L^{1}(0,T)$ and  for  each $t \in (0,T)$, there exists a function $w_{t} \in L^{2}(0,t;H)$ such that 
  \begin{align}
    \|w_{t}(s) - u(s)\|_{H} &\leq c_{2}\,|t - s| \bigl( 1 + \bigl| \varphi^{s} \bigl( u(s) \bigr) \bigr| \bigr)^{1/2}, \label{eq:A2-z-approx1}
    \\
    \varphi^{t}\bigl( w_{t}(s) \bigr) &\leq \varphi^{s}\bigl(u(s)\bigr) 
    + c_{2}\,|t-s| \bigl( 1+ \bigl| \varphi^{s} \bigl( u(s) \bigr) \bigr| \bigr) \label{eq:A2-z-approx2}
  \end{align}
  for a.e.~$s \in (0,t)$.  
\end{itemize}

Here, the conditions \textup{(A1)} and \textup{(A2)} are essentially different from each other. 
Indeed, the condition \textup{(A1)} only requires that for each 
$s,t \in [0,T]$ with $s \leq t$ and each $z_{s} \in D(\varphi^{s})$, we can choose a point $z_{s,t} \in D(\varphi^{t})$ satisfying \eqref{eq:A1-z-approx1} and \eqref{eq:A1-z-approx2}.  
This does not ensure the strong measurability of the mapping $t \mapsto z_{s,t}$.  
In contrast, the condition \textup{(A2)} guarantees that for each $u \in L^{2}(0,T;H)$ satisfying $\varphi^{\bullet} (u(\bullet)) \in L^{1}(0,T)$ and for each $t \in (0,T)$, there exists a function $w_{t} \in L^{2}(0,t;H)$ satisfying \eqref{eq:A2-z-approx1} and \eqref{eq:A2-z-approx2}.

\begin{remark}[A sufficient condition for \textup{(A1)} and  \textup{(A2)}] \label{Remark: Sufficient condition for A1 and A2}
Let $\Delta$ be defined by 
\begin{align*} 
  \Delta := \{(t,s) \in [0,T] \times [0,T] \colon t \geq s \}. 
\end{align*} 
We introduce a nonlocal version of the Kenmochi condition:  
\begin{itemize}
\item[\textup{$(A\varphi^{t})_{NL}$}]
%The real Hilbert space $H$ is separable. Furthermore, 
There exist a constant $C \in [0,\infty)$ and a Carath\'eodory function $\Psi \colon \Delta\times H \to H$, i.e., for each $w \in H$, the mapping $(t,s) \mapsto \Psi(\bullet,\bullet,w)$ is strongly measurable, and for~a.e.~$(t,s) \in \Delta$, $\Psi(t,s,\bullet)$ is continuous in $H$, such that 
\begin{align*} 
  \|\Psi(t,s,w) - w\|_{H}            &\leq C\,|t-s| \bigl(1 + |\varphi^{s}(w)|\bigr)^{1/2}, \\
  \varphi^{t}\bigl(\Psi(t,s,w)\bigr) &\leq \varphi^{s}(w) + C\,|t-s| \bigl(1 + |\varphi^{s}(w)|\bigr)
\end{align*}
for all $(t,s) \in \Delta$ and all $w \in H$.  
\end{itemize} 
Then \textup{$(A\varphi^{t})_{NL}$} implies both \textup{(A1)} and  \textup{(A2)} with $c_{1}=c_{2}=C$.  
Indeed, to verify \textup{(A1)}, for each $s,t \in [0,T]$ with $s \leq t$ and each $z_{s} \in D(\varphi^{s})$, we set $z_{s,t}:=\Psi(t,s,z_{s})$.
Then the above inequalities imply that the condition \textup{(A1)} is satisfied.  
To verify \textup{(A2)}, for each $u \in L^{2}(0,T;H)$ satisfying 
$\varphi^{\bullet}(u(\bullet)) \in L^{1}(0,T)$ and for each $t \in (0,T)$, we set  $w_{t}( \bullet ):=\Psi \bigl( t, \bullet ,u( \bullet ) \bigr)$.  
Then, since $\Psi$ is a Carath\'eodory function, the mapping $s \mapsto w_{t}(s)$ is strongly measurable in $(0,T)$, and the condition \textup{(A2)}  also follows from the same estimates. 
\end{remark}

Under these assumptions, we obtain the following existence result for  the case where $f \in L^{2}(0,T;H)$\textup{:} 

\begin{theorem}[Existence of strong solutions for $f \in L^{2}(0,T;H)$] \label{thm:main2} 
Let $(k,\ell) \in PC$ and assume that \textup{(A1)} and  \textup{(A2)} hold.  
Then, for every $u_{0} \in D(\varphi^{0})$ and $f \in L^{2}(0,T;H)$, the Cauchy problem \eqref{E:main-equation1} admits a unique strong solution $u \in L^{2}(0,T;H)$ on $[0,T]$. 
\end{theorem}

%%%%%%%%%%%%%%%%%%%%%%%%%%%
%%%%%%%%%%%%%%%%%%%%%%%%%%%
%%%%%%%%%%%%%%%%%%%%%%%%%%%
%%%%%%%%%%%%%%%%%%%%%%%%%%%
%%%%%%%%%%%%%%%%%%%%%%%%%%%
%%%%%%%%%%%%%%%%%%%%%%%%%%%
%%%%%%%%%%%%%%%%%%%%%%%%%%%
%%%%%%%%%%%%%%%%%%%%%%%%%%%

\section{Preliminaries} \label{Sec: Preliminaries}

\subsection{Maximal monotone operators and Subdifferential operators}  \label{subsec:subdifferential}

We briefly recall some basic  facts on  subdifferential operators and maximal monotone operators; 
see, e.g.,  \cite{B-Barbu-1976,B-BauschkeHeinzCombettes-2017,B-Brezis-1973,B-Cioranescu-1990,B-PapageorgiouGasinski-2006,B-PapageorgiouNikolaosRadulescuRepovs-2019,B-Roubivcek-2005,B-Showalter-1997}  for details.  

Let $\varphi: H \to (-\infty,\infty]$ be a proper (i.e., $\varphi \not\equiv \infty$), lower-semicontinuous, and convex functional with the \emph{effective domain},  
\begin{equation*}
D(\varphi):=\{w \in H \colon \varphi(w) < \infty\} \neq \emptyset.
\end{equation*}
We define the \emph{subdifferential operator} $\partial \varphi \colon H \to 2^{H}$ by 
\begin{equation*}
\partial \varphi(z):=\{\xi \in H \colon \varphi(v)-\varphi(z) \geq (\xi,v-z)_{H} \textup{ for all } v\in D(\varphi)\}
\end{equation*}
for $z \in H$. 
It is a fundamental result that $\partial \varphi$ is a maximal monotone operator (see \cite[Example~2.3.4]{B-Brezis-1973}, \cite[Chapter~II, Theorem~2.1]{B-Barbu-1976}, also \cite{B-PapageorgiouNikolaosRadulescuRepovs-2019,B-BauschkeHeinzCombettes-2017}). %\cite[Chapter~II~Theorem~1.2]{B-Barbu-1976}
Here,  an operator $A \colon H \to 2^{H}$ is said to be \emph{maximal monotone} (or $m$-accretive) 
if the following two conditions hold:
\begin{itemize} 
  \item[(i)]
  $(\xi_{1} - \xi_{2}, u_{1} - u_{2})_{H} \geq 0$ for  all  $(u_{1},\xi_{1}), (u_{2},\xi_{2}) \in G(A)$,   
  \item[(ii)] 
  $R(I+\lambda A) = H$ for every $\lambda \in (0,\infty)$. 
\end{itemize}
The condition~(i) is equivalent to the following one  (see, e.g., %\cite[Lemma~2.13, Proposition~20.2]{B-BauschkeHeinzCombettes-2017} 
\cite[Proposition~2.1]{B-Brezis-1973}, \cite[Chapter~II, Proposition~3.1]{B-Barbu-1976}, 
\cite[Proposition~3.3.4]{B-PapageorgiouGasinski-2006} and \cite{B-Showalter-1997,B-BauschkeHeinzCombettes-2017}):  
\begin{itemize}
  \item[(i)']
  For every $\lambda \in (0,\infty)$ and $(u_{1},\xi_{1}), (u_{2},\xi_{2}) \in G(A)$, it holds that 
  \begin{equation*} 
    \|(u_{1}+ \lambda \xi_{1}) - (u_{2}+ \lambda \xi_{2})\|_{H} \geq \|u_{1}-u_{2}\|_{H}.
  \end{equation*}
  \par\noindent 
  In particular, $(I+ \lambda A)^{-1}\colon H \to 2^{H}$ is a single-valued operator.  
\end{itemize}

The \emph{minimal section} $\mathring{A} \colon D(A) \subset H \to H$ of $A$ is defined by 
$\mathring{A}(w) := \operatorname*{argmin}_{\xi \in A(w)} \|\xi \|_{H}$ for $w \in D(A)$. 
Since, for each $w \in D(A)$, the set $A(w) \subset H$ is a weakly closed convex subset of $H$  (see, e.g., \cite[p.~28]{B-Brezis-1973}, \cite[Chapter~II, Proposition~3.5]{B-Barbu-1976}, \cite[Proposition~3.2.7]{B-PapageorgiouGasinski-2006} and \cite{B-BauschkeHeinzCombettes-2017,B-PapageorgiouNikolaosRadulescuRepovs-2019}),   
%\cite[Chapter~II, Proposition~3.5]{B-Barbu-1976}, 少し怪しい. s\cite[Proposition~2.6.5]{B-PapageorgiouNikolaosRadulescuRepovs-2019}), 
the minimizer is unique. 
Therefore $\mathring{A}(w)$ is well-defined for every $w \in D(A)$. 
Moreover, every maximal monotone operator is \emph{demiclosed}, that is, if 
$(u_{n},\xi_{n}) \in G(A)$, $u_{n} \to u$ strongly in $H$, and $\xi_{n} \rightharpoonup \xi$ weakly in $H$, 
then it follows that $(u,\xi) \in G(A)$ (see, e.g., \cite[Proposition~2.5]{B-Brezis-1973}, \cite[Chapter~II, Lemma~1.3]{B-Barbu-1976} and \cite{B-PapageorgiouGasinski-2006,B-BauschkeHeinzCombettes-2017}).   
% (see \cite[Proposition~20.37]{B-BauschkeHeinzCombettes-2017}). \cite[Proposition~3.2.15]{B-PapageorgiouGasinski-2006}, 

For each maximal monotone operator $A \colon H \to 2^{H}$, 
the \emph{resolvent} $J_{\lambda}^{A} \colon H \to H$ 
and the \emph{Yosida approximation} $A_{\lambda} \colon H \to H$ are defined by
\begin{equation*}
  J_{\lambda}^{A} := (I+\lambda A)^{-1}, 
  \quad
  A_{\lambda} := \frac{I - J_{\lambda}^{A}}{\lambda}
\end{equation*}
for $\lambda \in (0,\infty)$, respectively.  
The following properties are well known\textup{:}

\begin{proposition}[\textup{see, e.g., \cite[Propositions~2.2 and 2.6, Theorem~2.2]{B-Brezis-1973}, \cite[Chapter~II, Proposition~1.1]{B-Barbu-1976}, \cite[Theorem~3.2.38]{B-PapageorgiouGasinski-2006} and \cite{B-Showalter-1997,B-BauschkeHeinzCombettes-2017}}] 
\label{prop:JY}
Let $A \colon H \to 2^{H}$ be a maximal monotone operator. 
Then the following properties hold\/\textup{:}
\begin{enumerate}
  \item 
  For each $\lambda \in (0,\infty)$, the mapping $J_{\lambda}^{A} \colon H \to H$ is non-expansive \textup{(}i.e., $1$-Lipschitz continuous\textup{)}. Moreover, $J_{\lambda}^{A} w \to w$ strongly in $H$ as $\lambda \to 0_{+}$ for all $w \in \overline{D(A)}^{H}$. 
  \item 
  $A_{\lambda} \colon H \to H$ is maximal monotone. Moreover, $A_{\lambda}$ is Lipschitz continuous in $H$ with Lipschitz constant $1/\lambda$. 
  \item 
  $A_{\lambda}(w) \in A(J_{\lambda}^{A} w)$ for every $w \in H$ and  $\lambda \in (0,\infty)$.    
  \item  
  $\|A_{\lambda}(w)\|_{H} \leq \|\mathring{A}(w)\|_{H}$ for every $w \in D(A)$ and $\lambda \in (0,\infty)$.   
  Moreover, $A_{\lambda}(w) \to \mathring{A}(w)$ strongly in $H$ as $\lambda \to 0_{+}$ for all $w \in D(A)$.  
\end{enumerate}
\end{proposition}

We next define the \emph{Moreau--Yosida regularization} $\varphi_{\lambda}\colon H \to \mathbb{R}$ of $\varphi$ for $\lambda \in (0,\infty)$ by 
\begin{align*}
  \varphi_{\lambda}(w) 
  := \inf_{z \in H} \left( \frac{1}{2\lambda} \| w-z \|_{H}^{2} + \varphi(z) \right) 
   &= \frac{1}{2\lambda} \| w - J_{\lambda}^{\partial \varphi} w \|_{H}^{2} + \varphi ( J_{\lambda}^{\partial \varphi} w) \\
  &= \frac{\lambda}{2} \| (\partial \varphi)_{\lambda}(w) \|_{H}^{2} + \varphi ( J_{\lambda}^{\partial \varphi} w) 
\end{align*}
for $w \in H$. We recall the following well known result\textup{:} 

\begin{proposition}[\textup{see, e.g., \cite[Proposition~2.11]{B-Brezis-1973}, \cite[Chapter~II, Theorem 2.2]{B-Barbu-1976} and \cite{B-BauschkeHeinzCombettes-2017,B-PapageorgiouNikolaosRadulescuRepovs-2019}}]  %\textup{\cite[Proposition~2.9.12]{B-PapageorgiouNikolaosRadulescuRepovs-2019}} ]
\label{prop:MY}
Let $\varphi \colon H \to (-\infty,\infty]$ be a proper lower-semicontinuous convex functional, and let $\varphi_{\lambda}$ be its Moreau--Yosida regularization for $\lambda \in (0,\infty)$.   
Then the following assertions hold\/\textup{:} 
\begin{enumerate}
  \item 
  $\varphi_{\lambda}$ is convex and Fr\'echet differentiable in $H$, and its derivative satisfies $\partial \varphi_{\lambda} = (\partial \varphi)_{\lambda}$, that is, the derivative $\partial (\varphi_{\lambda})$ of $\varphi_{\lambda}$ coincides with the Yosida approximation $(\partial \varphi)_{\lambda}$ of $\partial \varphi$.  Hence, for simplicity, the notation $\partial \varphi_{\lambda}$ is used instead of both $\partial (\varphi_{\lambda})$ and $(\partial \varphi)_{\lambda}$ in what follows. 
  \item 
  It holds that $\varphi(J_{\lambda}^{\partial \varphi} w) \leq \varphi_{\lambda}(w) \leq \varphi(w)$ for every $w \in H$ and $\lambda \in (0,\infty)$.  
  \item 
  $\varphi_{\lambda}(w) \to \varphi(w)$ as $\lambda \to 0_{+}$ for every $w \in H$.  
\end{enumerate}
\end{proposition}

\subsection{Time-dependent subdifferential operators} \label{subsec:Time-dependent subdifferential}

In this subsection, we briefly summarize some properties of the time-dependent subdifferential operators under the assumption (A1); for details, we refer the reader to \cite{A-Kenmochi-1975,A-Yamada-1976,A-Kenmochi-1981,A-Otani-1993/94}.  
For each $t \in [0,T]$, let $\varphi^{t} \colon H \to (-\infty,\infty]$ be a proper lower-semicontinuous convex functional, and assume that (A1) holds.  
In what follows, for each $\lambda \in (0,\infty)$ and $t \in [0,T]$, the resolvent $J_{\lambda}^{\partial\varphi^{t}}$ of $\partial \varphi^{t}$ is denoted by $J_{\lambda}^{t} := (I+\lambda\,\partial\varphi^{t})^{-1}$.  
We collect below several basic properties.

\begin{proposition}[\textup{see \cite[Lemmas~1.5.1 and 1.5.4]{A-Kenmochi-1981}}] \label{prop:time-subdiff} 
Assume that \textup{(A1)} holds. Then there exists a constant $C \in [0,\infty)$ such that for every $\lambda \in (0,1)$, $z \in H$ and $t \in [0,T]$, the following inequalities hold\/\textup{:}
\begin{gather*}
  \varphi^{t}(z) \geq - C \bigl( \| z \|_{H} + 1 \bigr), 
  \quad 
  \|J_{\lambda}^{t} z\|_{H}
  \leq 
  C \bigl(1+\|z\|_{H}\bigr), 
  \\ 
  \|\partial \varphi^{t}_{\lambda}(z)\|_{H}
  \leq 
  \dfrac{C}{\lambda}\bigl(1+\|z\|_{H}\bigr).  
\end{gather*} 
\end{proposition}

From \cref{prop:time-subdiff,prop:MY}, there exists a constant 
$C \in [0,\infty)$ such that 
\begin{align*}
  \varphi^{t}(z) 
  \geq 
  \varphi_{\lambda}^{t}(z) 
  \geq 
  \varphi^{t}\bigl(J_{\lambda}^{t} z\bigr) 
  \geq 
  - C \bigl(\|z\|_{H} + 1\bigr)
\end{align*}
for all $\lambda \in (0,1)$, $z \in H$ and $t \in [0,T]$.  
Hence there exists a constant $D \in (0,\infty)$ such that
\begin{align}
  |\varphi^{t}(z)| 
  &\leq 
  \varphi^{t}(z) + D\bigl(\| z \|_{H} + 1\bigr),
  \label{eq:td-subdiff-varphi-bound}
  \\
  |\varphi_{\lambda}^{t}(z)| 
  &\leq 
  \varphi_{\lambda}^{t}(z) + D\bigl(\| z \|_{H} + 1\bigr)
\label{eq:td-subdiff-varphi-lambda-bound}
\end{align}
for all $\lambda \in (0,1)$, $z \in H$ and $t \in [0,T]$.

\begin{proposition}[\textup{see \cite[Lemmas~1.2.2~and~1.2.3, Theorem~1.5.1]{A-Kenmochi-1981}}] \label{prop:Phi}
Assume that \textup{(A1)} holds. Then for each $\lambda \in (0,1)$ and $v \in L^{1}(0,T;H)$, 
the maps $t \mapsto \varphi^{t}_{\lambda} \bigl( v(t) \bigr)$ and $t \mapsto \varphi^{t} \bigl(  v (t) \bigr)$ are measurable in $(0,T)$. 
Furthermore, the maps $t \mapsto J_{\lambda}^{t} \bigl(  v (t) \bigr)$ and $t \mapsto \partial \varphi_{\lambda}^{t} \bigl(  v (t) \bigr)$ are strongly measurable in $(0,T)$.  
\end{proposition}

\begin{proposition}[\textup{see \cite[Proposition~1.1]{A-Kenmochi-1975}, \cite[Lemmas~1.2.2~and~1.2.4, Theorem~1.5.1]{A-Kenmochi-1981}}] 
\label{prop:Phi-t-properties}  
Assume that \textup{(A1)} holds. 
 Let $\Phi \colon L^{2}(0,T;H) \to (-\infty,\infty]$ be a functional defined  by
\begin{equation*}
  \Phi(u) :=
  \begin{cases}
    \displaystyle \int_{0}^{T} \varphi^{t} \bigl( u(t) \bigr) \,\textup{d}t 
    & \textup{if } \varphi^{ \bullet } \bigl( u(\bullet) \bigr) \in L^{1}(0,T), \\[1.5ex]
    \infty & \textup{otherwise}
  \end{cases}
\end{equation*}
 for $u \in L^{2}(0,T;H)$. 
Then $\Phi$ is a proper, lower-semicontinuous, and convex functional with  the effective domain, 
\begin{align*}
D(\Phi) = \left\{ w \in L^{2}(0,T;H) \colon \varphi^{\bullet} \bigl( w( \bullet ) \bigr) \in L^{1}(0,T) \right\}. 
\end{align*}
Moreover, for each $\lambda \in (0,1)$, the Moreau--Yosida regularization 
$\Phi_{\lambda}$ of $\Phi$ is given by
\begin{equation*}
  \Phi_{\lambda}(u) = \int_{0}^{T} \varphi^{t}_{\lambda} \bigl( u(t) \bigr) \,\textup{d}t
  \quad \textup{for $u \in L^{2}(0,T;H)$.}
\end{equation*}
Furthermore,  for each $u \in L^{2}(0,T;H)$,  the subdifferential operator $\partial \Phi_{\lambda}$ satisfies 
\begin{align*} 
  [\partial \Phi_{\lambda}(u)](t)
  =
  \partial \varphi^{t}_{\lambda}\bigl(u(t)\bigr)
  \quad \textup{for a.e.~} t \in (0,T),
\end{align*}
 and  the corresponding resolvent $J_{\lambda}^{\partial\Phi}$ is given by
\begin{equation*} 
 [J_{\lambda}^{\partial \Phi}(u)](t) = J^{t}_{\lambda} \bigl( u(t) \bigr)
  \quad \textup{for a.e.~} t \in (0,T).
\end{equation*} 
The subdifferential operator $\partial \Phi$ can be characterized as 
\begin{equation*} %\label{prop:eq-Partial-Phi}
  \partial \Phi(u) 
  = 
  \Bigl\{ 
  w \in L^{2}(0,T;H) \colon  
  w(t) \in \partial \varphi^{t} \bigl( u(t) \bigr)
  \quad \textup{for a.e.~} t \in (0,T) 
  \Bigr\}
\end{equation*}
for $u \in L^{2}(0,T;H)$. 
\end{proposition}

\begin{proposition}[\textup{see \cite[Lemma~1.2.5 and Theorem~1.5.1]{A-Kenmochi-1981}} ] \label{prop:time-subdiff-energy}
Assume that \textup{(A1)} holds. Let $\lambda \in (0,1)$ and $u \in W^{1,2}(0,T;H)$. 
Then the mapping $t \mapsto \varphi^{t}_{\lambda} \bigl( u(t) \bigr)$ is  differentiable a.e.~in~$(0,T)$, and for every $s,t \in [0,T]$ with $s \leq t$,  it holds that 
\begin{equation*}
  \varphi^{t}_{\lambda} \bigl( u(t) \bigr) - \varphi^{s}_{\lambda} \bigl( u(s) \bigr)
  \leq \int_{s}^{t}  \dfrac{\textup{d}}{\textup{d} \tau}  \varphi^{\tau}_{\lambda} \bigl( u(\tau) \bigr) \,\textup{d}\tau.
\end{equation*}
Moreover, there exist a constant $c_{0} \in [0,1)$ and nonnegative functions $\eta_{1}, \eta_{2} \in L^{1}(0,T)$ such that 
\begin{align*} 
  \dfrac{\textup{d}}{\textup{d} t} \varphi_{\lambda}^{t} \bigl( u(t) \bigr) 
  &\leq     
  \bigl( \partial \varphi_{\lambda}^{t} \bigl( u(t) \bigr), u'(t) \bigr)_{H} 
  +
  c_{0} \bigl\|\partial \varphi_{\lambda}^{t} \bigl( u(t) \bigr) \bigr\|_{H}^{2} 
  \nonumber \\ & \quad + \eta_{1}(t) \bigl|\varphi_{\lambda}^{t} \bigl( u(t) \bigr) \bigr|   
  + \eta_{2}(t)
\end{align*} 
for~a.e.~$t \in (0,T)$. Here the constant $c_{0}$ and the functions $\eta_{1}, \eta_{2}$ are independent of both $\lambda$ and $u$. 
In particular,  it holds that   
\begin{align}
  %% 1行目（差分のエネルギー評価）
  &\varphi_{\lambda}^{t} \bigl( u(t) \bigr) - \varphi_{\lambda}^{0} \bigl( u(0) \bigr) 
  \leq 
  \int_{0}^{t}  \dfrac{\textup{d}}{\textup{d} \tau}  \varphi^{\tau}_{\lambda} \bigl( u(\tau) \bigr) \,\textup{d}\tau  
  \nonumber \\ &\quad \leq 
  \int_{0}^{t} \bigl( \partial \varphi_{\lambda}^{\tau} \bigl( u(\tau) \bigr), u'(\tau) \bigr)_{H} \,\textup{d}\tau 
  \nonumber \\
  %% 2行目（残りの積分項）
  &\quad \quad + \int_{0}^{t} \Bigl( 
  c_{0} \bigl\|\partial \varphi_{\lambda}^{\tau} \bigl( u(\tau) \bigr)\bigr\|_{H}^{2} 
  + \eta_{1}(\tau) \bigl|\varphi_{\lambda}^{\tau} \bigl( u(\tau) \bigr) \bigr| 
  + \eta_{2}(\tau) \Bigr)\,\textup{d}\tau
  \nonumber 
\end{align}
for all $t \in (0,T]$.  
\end{proposition}

\begin{comment}
\begin{remark}
In  \cite{A-Kenmochi-1981}, $\{ \varphi^{t}\}_{t \in [0,T]}$ is assumed to satisfy the weaker assumption: 
\begin{itemize}
  \item[] 
  For each $r \in (0,\infty)$, there are $\theta_{1,r} \in W^{1,2}(0,T)$ and $\theta_{2,r} \in W^{1,1}(0,T)$ with the following property: for each $s,t \in [0,T]$ with $s \leq t$ and each $z_{s} \in D(\varphi^{s})$ satisfying $\|z_{s}\|_{H} \leq r$, there exists $z_{s,t} \in D(\varphi^{t})$ such that 
  \begin{align*}
    \|z_{s,t} - z_{s}\|_{H} 
    &\leq |\theta_{1,r}(t) - \theta_{1,r}(s)| \,\bigl(1+|\varphi^{s}(z_{s})|^{1/2}\bigr), \\
    \varphi^{t}(z_{s,t}) 
    &\leq \varphi^{s}(z_{s}) + |\theta_{2,r}(t) - \theta_{2,r}(s)| \,\bigl(1+|\varphi^{s}(z_{s})|\bigr).
  \end{align*}
\end{itemize}   
\end{remark}
\end{comment}

%%%%%%%%%%%%%%%%%%%%%%%%%%%
%%%%%%%%%%%%%%%%%%%%%%%%%%%
%%%%%%%%%%%%%%%%%%%%%%%%%%%
%%%%%%%%%%%%%%%%%%%%%%%%%%%

%%the operators that will play a central role in our analysis, namely 
%the (classical) \emph{time-differential operator} and 

\subsection{Nonlocal and local time-differential operators} \label{subsec:time-nonlocal-ops}

In this subsection, we introduce  several nonlocal and local time-differential operators  
and recall some of their basic properties.  
For more general settings, see \cite{A-Gripenberg-1985,A-Clement-1981,A-Clement-1984,A-VergaraZacher-2008,A-Zacher-2008}.

 Let $(k,\ell) \in PC$.   
The \emph{nonlocal time-differential operator}  
$\mathcal{B} \colon D(\mathcal{B}) \subset L^{2}(0,T;H) \to L^{2}(0,T;H)$  
is defined by  
\begin{align*}
  D(\mathcal{B}) &:= \bigl\{ w \in L^{2}(0,T;H) \colon k*w \in W^{1,2}(0,T;H) \textup{ and } (k*w)(0) = 0 \bigr\}, \\ 
  \mathcal{B}(w) &:= \partial_{t}(k*w) \quad \textup{for } w \in D(\mathcal{B}).
\end{align*} 
Then $\mathcal{B}$ is a linear $m$-accretive operator (equivalently, a maximal monotone operator) in $L^{2}(0,T;H)$ (see \cite[Section~2]{A-VergaraZacher-2008}, also \cite{A-Clement-1984,A-Gripenberg-1985}), %\cite[Theorem~2.1]{A-Clement-1984}, \cite[Proposition~1]{A-Gripenberg-1985}
and for each $\lambda \in (0,\infty)$, the Yosida approximation $\mathcal{B}_{\lambda}$ is given by  
\begin{align*} %\label{eq:B-Yosida}
  \mathcal{B}_{\lambda}(w) 
  = \partial_{t}(k_{\lambda}*w) 
  \quad \text{for } w \in L^{2}(0,T;H),
\end{align*}
where $k_{\lambda} \in W^{1,1}_{\textup{loc}}([0,\infty))$ solves the Volterra equation, 
\begin{align} \label{eq:Volterra-klambda}
  \lambda k_{\lambda}(t) + (\ell * k_{\lambda})(t) = 1
  \quad \textup{for } t \in (0,\infty), 
  \quad 
  k_{\lambda}(0) = 1/\lambda.
\end{align}
It has been shown that $k_{\lambda}$ is a nonnegative and nonincreasing function (see, e.g.,~\cite[Theorem~2.2]{A-Clement-1981}). 
Furthermore, it is known that $k_{\lambda} \to k$ strongly in $L^{1}(0,T)$ as $\lambda \to 0_{+}$ (see, e.g., \cite[Lemma~3.1]{A-Gripenberg-1985}, \cite[Section~2]{A-Zacher-2008}). 
Moreover, the following chain-rule formula holds. 

\begin{proposition}[\textup{see \cite[Proposition 3.4]{A-Akagi-2019}, \cite[Proposition 4.2]{P-AkagiNakajima-2025}}] \label{P:frac_chain}
Let $\varphi: H \to (-\infty,\infty]$ be a proper {\rm (}i.e., $\varphi \not\equiv \infty${\rm )} lower-semicontinuous convex functional.  
Let $u \in L^{2}(0,T;H)$ and $u_{0} \in D(\varphi)$ be such that $\varphi\bigl( u(\bullet) \bigr) \in L^{1}(0,T)$, $u - u_{0} \in D(\mathcal{B})$ and $u(t) \in D(\partial \varphi)$ for a.e.~$t \in (0,T)$.  
Assume that there exists $g \in L^{2}(0,T;H)$ satisfying $g(t) \in \partial \varphi(u(t))$ for a.e.~$t \in (0,T)$.  
Then it holds that
\begin{equation*} 
\int_{0}^{t} \bigl( \mathcal{B}(u-u_{0})(\tau), g(\tau) \bigr)_{H} \, \textup{d}\tau
\geq \bigl[k * \bigl( \varphi \bigl( u \bigr) - \varphi(u_{0}) \bigr)\bigr](t)
\end{equation*}
for a.e.~$t \in (0,T)$.  
In particular, if we choose $\varphi( \bullet ) = \tfrac{1}{2}\| \bullet  \|_{H}^{2}$, then for every $w \in L^{2}(0,T;H)$ and $w_{0} \in H$ with $w - w_{0} \in D(\mathcal{B})$, 
it holds that 
\begin{equation*}
\int_{0}^{t} \bigl( \mathcal{B}(w-w_{0})(\tau), w(\tau) \bigr)_{H} \, \textup{d}\tau 
\geq  \dfrac{1}{2} 
  \bigl[
    k * (\|w\|_{H}^{2} - \|w_{0}\|_{H}^{2})
  \bigr](t)
\end{equation*}
for a.e.~$t \in (0,T)$.
\end{proposition} 
Let $\mathcal{A} \colon D(\mathcal{A}) \subset L^{2}(0,T;H) \to L^{2}(0,T;H)$  be  
the classical  (local)  time-differential operator defined by  
\begin{align*}
  D(\mathcal{A}) := \{ w \in W^{1,2}(0,T;H) \colon w(0)=0 \}, 
  \quad  
  \mathcal{A}(w) := \partial_{t} w \quad \textup{for } w \in D(\mathcal{A}).
\end{align*} 
Then we observe that $D(\mathcal{A}) \subset D(\mathcal{B})$. 
We recall that for each  $u \in D(\mathcal{A})$, the following inequality holds:
\begin{equation}  \label{E:AB-energy-estimate}
  \int_{0}^{t} \bigl( \mathcal{B}(u)(\tau), \mathcal{A}(u)(\tau) \bigr)_{H} \, \textup{d}\tau
  \geq 
  \dfrac{1}{2} \bigl( \ell * \|\mathcal{B}(u)\|_{H}^{2} \bigr)(t)
\end{equation}
for~a.e.~$t \in (0,T)$ (see \cite[Corollary 3.6]{A-Akagi-2019}). 
Therefore, for every $u \in D(\mathcal{A})$,  
H\"older's inequality yields   
\begin{align} 
  \dfrac{1}{2} \bigl( \ell * \|\mathcal{B}(u)\|_{H}^{2} \bigr)(t) 
  &\leq 
  \bigl( \mathcal{B}(u),\mathcal{A}(u)\bigr)_{L^{2}(0,t;H)}
  \notag  \\  
  &\leq    
  \| \mathcal{B}(u)\|_{L^{2}(0,T;H)} \| \mathcal{A}(u)\|_{L^{2}(0,T;H)}   \label{eq:AB-Holder}
\end{align}
for~a.e.~$t \in (0,T)$.

%%%%%%%%%%%%%%%%%%%%%%%%%%%
%%%%%%%%%%%%%%%%%%%%%%%%%%%
%%%%%%%%%%%%%%%%%%%%%%%%%%%
%%%%%%%%%%%%%%%%%%%%%%%%%%%
%%%%%%%%%%%%%%%%%%%%%%%%%%%
%%%%%%%%%%%%%%%%%%%%%%%%%%%
%%%%%%%%%%%%%%%%%%%%%%%%%%%
%%%%%%%%%%%%%%%%%%%%%%%%%%%

\section{Some devices} \label{Sec: Some devices}

%%%%%%%%%%%%%%%%%%%%%%%%%%%
%%%%%%%%%%%%%%%%%%%%%%%%%%%
%%%%%%%%%%%%%%%%%%%%%%%%%%%
%%%%%%%%%%%%%%%%%%%%%%%%%%%
\subsection{Modified Gronwall-type lemma}

In this subsection, we prove a Gronwall-type lemma  (see \cref{prop:ineq} below), which will be used later in the proof of our main results. 
To this end, we first establish the following auxiliary inequality\textup{:}   

%%%%%%%%%
%%%%%%%%%
%%%%%%%%%
%%%%%%%%%

\begin{lemma}\label{lem:weighted}
Let $g_{1}, g_{2} \in L^{1}(0,T)$ be nonnegative functions, and let $f \in L^{\infty}(0,T)$.  
Then, for every $ \beta  \in (0,\infty)$, it holds that 
\begin{align*}
  &\int_{0}^{t} g_{1}(s) f(s)\,\textup{d}s +  (g_{2} * f)(t) \\
  &\quad \leq 
  \|f\|_{L^{\infty}_{ \beta }(0,T)}
  \exp\!\left(  \beta  \int_{0}^{t} g_{1}(s)\,\textup{d}s +  \beta  t \right)
  \left( \dfrac{1}{ \beta } + \int_{0}^{t} g_{2}(s) \,  \textup{e}^{-  \beta  s}  \,\textup{d}s \right)
\end{align*}
for~a.e.~ $t \in (0,T)$, 
where the weighted norm  $\| \bullet \|_{L^{\infty}_{ \beta }(0,T)} \colon L^{\infty}(0,T) \to [0,\infty)$ is defined by
\begin{equation*}
  \|f\|_{L^{\infty}_{ \beta  }(0,T)} 
  := 
  \left\| 
    f(\bullet)\,
    \exp\!\left(
      -  \beta  \int_{0}^{ \bullet } g_{1}(\tau)\,\textup{d}\tau -  \beta \bullet  
    \right) 
  \right\|_{L^{\infty}(0,T)}.
\end{equation*}
\end{lemma}

\begin{proof}
We estimate the two terms  of the left-hand side  separately.  
We first consider the integral term.  
From the definition of the weighted norm, we have
\begin{equation}  \label{eq:weighted-estimate}
  f(s) 
  \leq 
  \|f\|_{L^{\infty}_{ \beta }(0,T)} 
  \exp\!\left(  \beta  \int_{0}^{s} g_{1}(\tau)\,\textup{d}\tau +  \beta  s \right)
\end{equation}
for~a.e.~$s \in (0,T)$. 
Therefore we observe that  
\begin{align} 
  &\int_{0}^{t} g_{1}(s) f(s)\,\textup{d}s 
  \notag \\ 
  &\quad \leq % [1 line]
  \|f\|_{L^{\infty}_{  \beta  }(0,T)}
  \int_{0}^{t} g_{1}(s) 
  \exp\!\left(  \beta  \int_{0}^{s} g_{1}(\tau)\,\textup{d}\tau +  \beta  s \right)\,\textup{d}s 
  \notag \\ 
  &\quad \leq % [2–4 line]
  \|f\|_{L^{\infty}_{  \beta  }(0,T)}\,  \textup{e}^{  \beta  t} 
  \int_{0}^{t} g_{1}(s) 
  \exp\!\left(  \beta  \int_{0}^{s} g_{1}(\tau)\,\textup{d}\tau \right)\,\textup{d}s 
  \notag \\ 
  &\quad = 
  \|f\|_{L^{\infty}_{  \beta  }(0,T)}\,   \textup{e}^{  \beta  t} 
  \int_{0}^{t} 
  \dfrac{1}{  \beta  }\,\partial_{s}
  \left[ 
    \exp\!\left(  \beta  \int_{0}^{s} g_{1}(\tau)\,\textup{d}\tau \right) 
  \right] \,\textup{d}s 
  \notag \\ 
  &\quad = 
  \dfrac{1}{  \beta  }\, \|f\|_{L^{\infty}_{  \beta  }(0,T)}\,  \textup{e}^{  \beta  t} 
  \left[ 
    \exp\!\left(  \beta  \int_{0}^{t} g_{1}(\tau)\,\textup{d}\tau \right) - 1 
  \right] 
  \notag \\ 
  &\quad \leq % [5 line]
  \dfrac{1}{  \beta  }\, \|f\|_{L^{\infty}_{  \beta  }(0,T)} 
  \exp\!\left(  \beta  \int_{0}^{t} g_{1}(\tau)\,\textup{d}\tau +  \beta  t \right)
  \label{eq:aux-ineq-1}
\end{align}
for~a.e.~$t \in (0,T)$. 
Next, we turn to the convolution term.  
 From \eqref{eq:weighted-estimate},  we deduce that  
\begin{align} 
  &  (g_{2} * f)(t)  
  = \int_{0}^{t} g_{2}(t-s) f(s)\,\textup{d}s 
  \notag \\ 
  &\quad \leq % [1 line]
  \|f\|_{L^{\infty}_{  \beta  }(0,T)}
  \int_{0}^{t} g_{2}(t-s)
  \exp\!\left(  \beta  \int_{0}^{s} g_{1}(\tau)\,\textup{d}\tau +  \beta  s \right)\,\textup{d}s  
  \notag \\ 
  &\quad \leq % [2–3 line]
  \|f\|_{L^{\infty}_{  \beta  }(0,T)} 
  \exp\!\left(  \beta  \int_{0}^{t} g_{1}(\tau)\,\textup{d}\tau \right)
  \int_{0}^{t} g_{2}(t-s) \,   \textup{e}^{  \beta  s}   \,\textup{d}s  
  \notag \\ 
  &\quad = % [4–6 line]
  \|f\|_{L^{\infty}_{  \beta  }(0,T)} 
  \exp\!\left(  \beta  \int_{0}^{t} g_{1}(\tau)\,\textup{d}\tau \right)   \textup{e}^{  \beta  t} 
  \notag \\ 
  & \qquad \times  
  \int_{0}^{t} g_{2}(t-s)  \textup{e}^{- \beta(t-s)} \,\textup{d}s 
  \notag \\ 
  &\quad = 
  \|f\|_{L^{\infty}_{ \beta }(0,T)} 
  \exp\!\left( \beta \int_{0}^{t} g_{1}(\tau)\,\textup{d}\tau + \beta t \right)
  \int_{0}^{t} g_{2}(s) \,   \textup{e}^{-  \beta  s}   \,\textup{d}s
  \label{eq:aux-ineq-2}
\end{align}
for a.e. $t \in (0,T)$.  
Combining \eqref{eq:aux-ineq-1} and \eqref{eq:aux-ineq-2}, we obtain  
\begin{align*}
  &\int_{0}^{t} g_{1}(s) f(s)\,\textup{d}s +  ( g_{2} * f)(t)   
  \\ 
  &\quad \leq % [1–2 line]
  \|f\|_{L^{\infty}_{  \beta  }(0,T)}
  \exp\!\left(  \beta  \int_{0}^{t} g_{1}(\tau)\,\textup{d}\tau +  \beta  t \right)
  \left( \dfrac{1}{  \beta  } + \int_{0}^{t} g_{2}(s) \,   \textup{e}^{-  \beta  s}  \,\textup{d}s \right)
\end{align*}
for a.e. $t \in (0,T)$, which is the desired inequality.  
\end{proof}

Applying \cref{lem:weighted},  we can prove that the following Volterra integral equation admits a unique solution\textup{:} 

%%%%%%%%%%%%%%%%%%%
%%%%%%%%%%%%%%%%%%%
%%%%%%%%%%%%%%%%%%%
\begin{proposition}\label{L:Volterra}
Let $X$ be a real Banach space. Let $g_{1} \in L^{\infty}(0,T)$ and let $g_{2}, g_{3} \in L^{1}(0,T)$.   
Then there exists a unique solution $w \in L^{\infty}(0,T;X)$ to the Volterra integral equation,
\begin{equation*}
  w(t) = g_{1}(t) + \int_{0}^{t} g_{2}(s) w(s) \,\textup{d}s + (g_{3} * w)(t)  
  \quad \textup{for $t \in (0,T)$.}
\end{equation*}
\end{proposition}

\begin{proof}
For each $\beta \in (0,\infty)$, set $\mathfrak{X}_{\beta} := L^{\infty}(0,T;X)$ equipped with a norm given by
\begin{equation*}
  \|u\|_{\mathfrak{X}_{\beta}} := \left\| u(\bullet)\,
  \exp \!\left(-\beta \int_{0}^{\bullet} |g_{2}(\tau)|  \,\textup{d}\tau - \beta \bullet \right)
  \right\|_{L^{\infty}(0,T;X)} 
\end{equation*}
for $u \in \mathfrak{X}_{\beta}$. 
Then $( \mathfrak{X}_{\beta} , \|\bullet\|_{ \mathfrak{X}_{\beta} })$ is a Banach space for all $\beta \in (0,\infty)$, and $\|\bullet\|_{ \mathfrak{X}_{\beta} }$ is equivalent to the standard norm (without weight) of $L^{\infty}(0,T;X)$.
We define a map $\Lambda \colon L^{\infty}(0,T;X) \to L^{\infty}(0,T;X)$ by
\begin{equation*}
  [\Lambda(u)](t) := g_{1}(t) + \int_{0}^{t} g_{2}(s) u(s) \,\textup{d}s + (g_{3} * u)(t)
  \quad \textup{for~a.e.~$t \in (0,T)$.}
\end{equation*}

We claim that there exist $\beta_{0} \in (0,\infty)$ and $\kappa_{0} \in (0,1)$ such that
\begin{equation*}
  \|\Lambda(u_{1}) - \Lambda(u_{2})\|_{ \mathfrak{X}_{\beta_{0}} } 
  \leq \kappa_{0}\, \|u_{1} - u_{2}\|_{ \mathfrak{X}_{\beta_{0}} }
\end{equation*}
for all $u_{1}, u_{2} \in \mathfrak{X}_{\beta_{0}}$, i.e., $\Lambda$ is a contraction on $\mathfrak{X}_{\beta_{0}}$. 
Indeed, let $\beta \in (0,\infty)$ and take arbitrary $u_{1}, u_{2} \in \mathfrak{X}_{\beta}$. 
Using \cref{lem:weighted}, we have
\begin{align*}
%% 1 line
  &\|\Lambda(u_{1})(t) - \Lambda(u_{2})(t)\|_{X} \\
%% 2 line
  &\quad = 
  \left\|\int_{0}^{t} g_{2}(s)\,(u_{1}-u_{2})(s) \,\textup{d}s 
   + [g_{3} * (u_{1}-u_{2})](t)\right\|_{X} \\
%% 3 line
  &\quad \leq 
  \int_{0}^{t} |g_{2}(s)| \, \bigl\|(u_{1}-u_{2})(s)\bigr\|_{X} \,\textup{d}s 
  + [\,|g_{3}| * \|u_{1}-u_{2}\|_{X}\,](t)   \\
%% 4–5 line 
  &\quad \leq
  \|u_{1}-u_{2}\|_{ \mathfrak{X}_{\beta} } \exp \!\left( \beta \int_{0}^{t} |g_{2}(\tau)|\,\textup{d}\tau + \beta t\right)
  \left( \dfrac{1}{\beta} + \int_{0}^{t} |g_{3}(s)| \, \textup{e}^{- \beta s} \,\textup{d}s \right)
\end{align*}
for a.e.\ $t \in (0,T)$. 
Let $\kappa_{0} \in (0,1)$ be fixed. Then we can take a constant $\beta_{0} \in (0,\infty)$ sufficiently large such that
\begin{equation*}
  \dfrac{1}{\beta_{0}} + \int_{0}^{T} |g_{3}(s)| \, \textup{e}^{-\beta_{0} s} \,\textup{d}s \leq \kappa_{0}.
\end{equation*}
Then it follows that  
\begin{align*}
%% 1 line（重み付き評価に戻す）
 &\|\Lambda(u_{1})(t) - \Lambda(u_{2})(t)\|_{X} \, 
   \exp \!\left(- \beta_{0} \int_{0}^{t} |g_{2}(\tau)|\,\textup{d}\tau - \beta_{0} t\right)\\ 
%% 2–3 line（ノルムの定義を適用, $\kappa_{0}$ を用いた最終評価）
 &\quad \leq
  \|u_{1}-u_{2}\|_{ \mathfrak{X}_{\beta_{0}} }
  \left( \dfrac{1}{\beta_{0}} + \int_{0}^{t} |g_{3}(s)| \textup{e}^{ - \beta_{0} s} \,\textup{d}s \right) 
  \leq 
  \kappa_{0}\, \|u_{1}-u_{2}\|_{ \mathfrak{X}_{\beta_{0}} } 
\end{align*}
for a.e.\ $t \in (0,T)$. 
Taking the supremum over $t \in (0,T)$, we obtain    
\begin{equation*}
\|\Lambda(u_{1}) - \Lambda(u_{2})\|_{ \mathfrak{X}_{\beta_{0}} } \leq \kappa_{0}\, \|u_{1}-u_{2}\|_{ \mathfrak{X}_{\beta_{0}} },
\end{equation*}
which shows that $\Lambda$ is a contraction on $\mathfrak{X}_{\beta_{0}}$.  
Due to Banach's fixed point theorem, there exists a unique fixed point $w \in \mathfrak{X}_{\beta_{0}} = L^{\infty}(0,T;X)$ such that
$\Lambda(w) = w$.  
This completes the proof.  
\end{proof}

%%%%%%%%%
%%%%%%%%%
%%%%%%%%%
%%%%%%%%%

We are now in a position to derive a Gronwall-type lemma. 

%%%%%%%%%
%%%%%%%%%
%%%%%%%%%
%%%%%%%%%
\begin{proposition}\label{prop:ineq}
Let $g_{1} \in L^{\infty}(0,T)$, and let $g_{2}, g_{3} \in L^{1}(0,T)$ be nonnegative functions. 
Suppose that $f \in L^{\infty}(0,T)$ satisfies
\begin{equation}\label{eq:ineq1}
 f(t) \leq g_{1}(t) + \int_{0}^{t} g_{2}(s) f(s) \,\textup{d}s + (g_{3} * f)(t) 
 \quad \textup{for a.e.~} t \in (0,T).
\end{equation}
Then $f(t) \leq G(t)$ for~a.e.~$t \in (0,T)$, where $G \in L^{\infty}(0,T)$ denotes the unique solution of the Volterra integral equation, 
\begin{equation}\label{eq:ineq2}
 G(t) = g_{1}(t) + \int_{0}^{t} g_{2}(s) G(s) \,\textup{d}s + (g_{3} * G)(t)  
 \quad \textup{for $t \in (0,T)$}.
\end{equation}
\end{proposition} 

\begin{remark} 
By \cref{L:Volterra}, the Volterra integral equation \eqref{eq:ineq2} admits a unique solution $G \in L^{\infty}(0,T)$. 
In particular, we set a constant $D_{T} := \|G\|_{L^{\infty}(0,T)} < \infty$, which depends on $g_{1}$, $g_{2}$, $g_{3}$ and $T$. Then it holds that $f(t) \leq D_{T}$ for~a.e.~$t \in (0,T)$.
\end{remark}

\begin{proof}
Subtracting \eqref{eq:ineq2} from \eqref{eq:ineq1}, we have
\begin{align*}
 (f-G)(t) 
 &\leq 
 \int_{0}^{t} g_{2}(s) (f-G)(s)\,\textup{d}s + [g_{3} * (f-G)](t)\\
 &\leq 
 \int_{0}^{t} g_{2}(s) (f-G)_{+}(s)\,\textup{d}s + [g_{3} * (f-G)_{+}](t) 
\end{align*}
for a.e.~$t \in (0,T)$, where $(f-G)_{+} := \max \{0, f-G\} \in L^{\infty}(0,T)$. 
Here we used the nonnegativity of $g_{2}$ and $g_{3}$. 
Thus it holds that 
\begin{align} 
  (f-G)_{+}(t) 
  &= \max \left\{ 0, (f-G)(t) \right\}  \notag
  \\ &\leq 
  \int_{0}^{t} g_{2}(s) (f-G)_{+}(s)\,\textup{d}s + [g_{3} * (f-G)_{+}](t) \label{eq:ineq3} 
\end{align}
for~a.e.~$t \in (0,T)$. 
We now claim that $(f-G)_{+}(t) = 0$ for~a.e.~$t \in (0,T)$. 
It suffices to show that there exists a constant $\beta_{0} \in (0,\infty)$ such that $\|(f-G)_{+} \|_{L^{\infty}_{\beta_{0}}(0,T)} = 0$, where we see  
\begin{align*}
\|(f-G)_{+} \|_{L^{\infty}_{\beta}(0,T)} 
:=
\left\| (f-G)_{+}(\bullet)\,
\exp \!\left( - \beta \int_{0}^{\bullet} g_{2}(\tau)\,\textup{d}\tau - \beta \bullet \right)
\right\|_{L^{\infty}(0,T)}  
\end{align*}
for each $\beta \in (0,\infty)$.  
To see this, let $\beta \in (0,\infty)$ be a constant which will be determined later. 
Then, we derive from \eqref{eq:ineq3} and \cref{lem:weighted} that
\begin{align}
 &(f-G)_{+}(t) 
 \leq 
 \int_{0}^{t} g_{2}(s) (f-G)_{+}(s)\,\textup{d}s + [g_{3} * (f-G)_{+}](t) \notag \\
 &\quad \leq 
 \|(f-G)_{+} \|_{L^{\infty}_{\beta}(0,T)} 
 \exp \!\left( \beta \int_{0}^{t} g_{2}(s)\,\textup{d}s + \beta t \right) \notag \\ 
 &\quad\quad \times 
 \left( \dfrac{1}{\beta} + \int_{0}^{t} g_{3}(s) \, \textup{e}^{- \beta s} \,\textup{d}s \right)  \label{eq:ineq-weighted}
\end{align}
for~a.e.~$t \in (0,T)$. 
Let $\kappa_{0} \in (0,1)$ be fixed. 
Then we can take a constant $\beta_{0} \in (0,\infty)$ large enough such that 
\begin{equation*}
  \dfrac{1}{\beta_{0}} + \int_{0}^{T} g_{3}(s) \, \textup{e}^{- \beta_{0} s}  \,\textup{d}s \leq \kappa_{0}. 
\end{equation*}
Hence it follows from \eqref{eq:ineq-weighted} that 
\begin{align*}
  &(f-G)_{+}(t)\,
  \exp \! \left( -\beta_{0} \int_{0}^{t} g_{2}(\tau)\,\textup{d}\tau - \beta_{0} t \right) \\
  &\quad \leq 
  \|(f-G)_{+}\|_{L^{\infty}_{\beta_{0}}(0,T)} 
  \left( \dfrac{1}{\beta_{0}} + \int_{0}^{t} g_{3}(s) \, \textup{e}^{- \beta_{0} s} \,\textup{d}s  \right) \\
  &\quad \leq 
  \kappa_{0} \|(f-G)_{+}\|_{L^{\infty}_{\beta_{0}}(0,T)} 
\end{align*}
for~a.e.~$t \in (0,T)$.  
Taking the supremum over $t \in (0,T)$ yields 
\begin{align*}
\|(f-G)_{+}\|_{L^{\infty}_{\beta_{0}}(0,T)} 
  \leq 
  \kappa_{0} \|(f-G)_{+}\|_{L^{\infty}_{\beta_{0}}(0,T)}. 
\end{align*} 
Since $\kappa_{0} \in (0,1)$, we conclude that
$\|(f-G)_{+}\|_{L^{\infty}_{\beta_{0}}(0,T)} = 0$.
Thus $(f-G)_{+}(t)=0$ for~a.e.~$t \in (0,T)$, and hence
$f(t) \leq G(t)$ for~a.e.~$t \in (0,T)$.
This completes the proof.
\end{proof}

%%%%%%%%%%%%%%%%%%%%%%%%%%%%
%%%%%%%%%%%%%%%%%%%%%%%%%%%%
%%%%%%%%%%%%%%%%%%%%%%%%%%%%
%%%%%%%%%%%%%%%%%%%%%%%%%%%%
%%%%%%%%%%%%%%%%%%%%%%%%%%%%
%%%%%%%%%%%%%%%%%%%%%%%%%%%%
%%%%%%%%%%%%%%%%%%%%%%%%%%%%

\subsection{Fractional chain-rule formulae for time-dependent subdifferential operators}
\label{subsec:nonlocal-chainrule-time-dependent}

In this subsection, under assumptions \textup{(A1)} and \textup{(A2)}, we establish two fractional chain-rule formulae for time-dependent subdifferential operators (see \cref{lem:chainrule1,lem:chainrule2} below).      
We begin with the nonlocal chain-rule formula for Sobolev-regular kernels.

\begin{lemma}[Nonlocal chain-rule formula  for regular kernels] \label{lem:chainrule1} 
Let $k \in W^{1,1}(0,T)$ be a nonnegative and nonincreasing function.  
For each $t \in [0,T]$, let $\varphi^{t} \colon H \to (-\infty,\infty]$ be a proper lower-semicontinuous convex functional, and suppose that \textup{(A1)} and  \textup{(A2)} hold. 
Let $u_{0} \in D(\varphi^{0})$ and let $u,g \in L^{2}(0,T;H)$ be such that 
$\varphi^{ \bullet }(u(  \bullet )) \in L^{1}(0,T)$,  $u(t) \in D(\partial \varphi^{t})$ and 
$g(t) \in \partial \varphi^{t} \bigl( u(t) \bigr)$ for~a.e.~$t \in (0,T)$.   
Then there exists a constant $c \in [0,\infty)$ independent of $u$, $g$, and $k$ such that, for all $\varepsilon \in (0,1)$, the following inequality holds\/\textup{:}  
\begin{align} 
  &\bigl(\partial_{t}[k*(u-u_{0})](t),\, g(t)\bigr)_{H}   \notag \\
  &\quad \geq %%%%% [2nd–4th lines]
   \partial_{t}\bigl[k*\varphi^{\bullet}(u(\bullet))\bigr](t)  
  - k(t) \varphi^{0}(u_{0}) 
  - \varepsilon\,  c  \,\|k\|_{L^{1}(0,T)}\,\|g(t)\|_{H}^{2}  \notag \\
  &\qquad - 
   \dfrac{c}{\varepsilon} \|k\|_{L^{1}(0,T)}   \bigl(1+ |\varphi^{0}(u_{0})| \bigr)  \notag \\
  &\qquad - 
  \dfrac{c}{\varepsilon} \int_{0}^{t}\!  \bigl(-s k'(s)\bigr)  \, |\varphi^{\,t-s}(u(t-s))| \,\mathrm{d}s \label{eq:device-1}
\end{align} 
for~a.e.~$t \in (0,T)$, and moreover, the following inequality holds\/\textup{:}
\begin{align} 
  & %%%%% [1st line]
  \int_{0}^{t}
    \bigl(\partial_{t}[k*(u-u_{0})](s),\, g(s)\bigr)_{H}
    \,\textup{d}s  \notag \\
  &\quad \geq  %%%%% [8th–9th lines]
   \bigl[k*\varphi^{\bullet}(u(\bullet)) \bigr](t)  - \varphi^{0}(u_{0}) \int_{0}^{t}k(s) \, \textup{d}s  
  - \varepsilon\,  c  \,\|k\|_{L^{1}(0,T)}
    \int_{0}^{t}\!\|g(s)\|_{H}^{2}\,\textup{d}s  \notag \\
  &\qquad
  - \dfrac{c}{\varepsilon}  t  \,\|k\|_{L^{1}(0,T)}\bigl(1+|\varphi^{0}(u_{0})|\bigr)  \notag \\
  &\qquad
  - \dfrac{c}{\varepsilon} \|k\|_{L^{1}(0,T)} 
    \int_{0}^{t} |\varphi^{s}(u(s))| \,\textup{d}s \label{eq:device-2}
\end{align}
for  all~$t \in (0,T]$.  
\end{lemma}

\begin{proof}  
 Since $u_{0} \in D(\varphi^{0})$ and the assumption \textup{(A1)} holds, there is a constant $c_{1} \in [0,\infty)$ satisfying the following property\textup{:} for each $t \in [0,T]$, there exists $z_{t} \in D(\varphi^{t})$ such that  
\begin{align}
  \| z_{t} - u_{0} \|_{H}
  &\leq 
  c_{1} t \bigl( 1 + |\varphi^{0}(u_{0})| \bigr)^{1/2},
  \label{eq:z-approx-1}
  \\
  \varphi^{t}(z_{t})
  &\leq 
  \varphi^{0}(u_{0})
  + c_{1} t \bigl( 1 + |\varphi^{0}(u_{0})| \bigr).
  \label{eq:z-approx-2}
\end{align} 
 Similarly, from the assumption \textup{(A2)}, there is a constant $c_{2} \in [0,\infty)$ satisfying the following property\textup{:} for each $t \in (0,T)$, there exists  $w_{t} \in L^{2}(0,t;H)$ such that  
\begin{align} 
  \|w_{t}(s)-u(s)\|_{H} &\leq c_{2}|t-s|\bigl(1+ \bigl| \varphi^{s} \bigl( u(s) \bigr) \bigr| \bigr)^{1/2}, \label{eq:w-approx-1}
  \\
  \varphi^{t}\bigl(w_{t}(s)\bigr) &\leq \varphi^{s} \bigl( u(s) \bigr)  
  + c_{2}|t-s|\bigl(1+ \bigl|\varphi^{s} \bigl(u(s) \bigr) \bigr| \bigr)  \label{eq:w-approx-2}
\end{align}
for~a.e.~$s \in (0,t)$.  
By virtue of the fundamental theorem of calculus and direct computations, we deduce that 
\begin{align*}
  & %%%%% [1行目]
  \bigl(\partial_{t}[k*(u-u_{0})](t),\, g(t)\bigr)_{H} 
  \\
  &\quad = %%%%% [2行目]
  \bigl(k(0)\, (u(t)-u_{0}) + [k'*(u-u_{0})](t),\, g(t)\bigr)_{H}    
  \\
  &\quad = %%%%% [3,4,5行目]
  k(0)\,\varphi^{t}(u(t)) 
  + \bigl[ k'*\varphi^{ \bullet }(u(  \bullet )) \bigr](t)   
  \\
  &\qquad 
  + k(0)\Bigl[ \bigl( u(t)-u_{0},\, g(t) \bigr)_{H} -\varphi^{t}(u(t))\Bigr] 
  \\
  &\qquad 
  + \int_{0}^{t} k'(s)\Bigl[ \bigl( u(t-s)-u_{0},\, g(t) \bigr)_{H} - \varphi^{t-s}(u(t-s))\Bigr]\,\textup{d}s 
  \\
  &\quad = %%%%% [6,7,8行目]
  \partial_{t} \bigl[ k*\varphi^{ \bullet }(u( \bullet )) \bigr](t) 
  \\
  &\qquad 
  + \left(k(t)-\int_{0}^{t} k'(s)\,\textup{d}s\right) \Bigl[ \bigl( u(t)-u_{0},\, g(t) \bigr)_{H}-\varphi^{t}(u(t))\Bigr]
  \\
  &\qquad 
  + \int_{0}^{t}  \bigl( -k'(s) \bigr)   \Bigl[ \bigl(u_{0}-u(t-s),\, g(t) \bigr)_{H} + \varphi^{t-s}(u(t-s))\Bigr]\,\textup{d}s 
  \\
  &\quad = %%%%% [9,10行目]
  \partial_{t} \bigl[ k*\varphi^{ \bullet }(u( \bullet )) \bigr](t)  
  + k(t) \Bigl[ \bigl( u(t)-u_{0},\, g(t) \bigr)_{H}-\varphi^{t}(u(t))\Bigr]
  \\
  &\qquad 
  + \int_{0}^{t}  \bigl( -k'(s) \bigr)  \, \Bigl[ \bigl(u(t) - u(t-s),\, g(t) \bigr)_{H}  - \varphi^{t} (u(t))  + \varphi^{t-s}(u(t-s))\Bigr]\,\textup{d}s 
  \\
  &\quad = %%%%% [11,12,13,14行目] 
  \partial_{t} \bigl[ k*\varphi^{ \bullet }(u( \bullet )) \bigr](t) 
  + k(t)\Bigl[ \bigl( u(t)-z_{t},\, g(t) \bigr)_{H} -\varphi^{t}(u(t)) \Bigr] 
  \\
  &\qquad +
  k(t)  \bigl( z_{t} - u_{0},\, g(t) \bigr)_{H}  
  \\
  &\qquad 
  + \int_{0}^{t}  \bigl( -k'(s) \bigr)  \Bigl[ \bigl( u(t)-w_{t}(t-s),\, g(t) \bigr)_{H} - \varphi^{t}(u(t))+\varphi^{t-s}(u(t-s))\Bigr]\,\textup{d}s
  \\
  &\qquad 
  + \int_{0}^{t}  \bigl( -k'(s) \bigr)  \bigl( w_{t}(t-s) - u(t-s),\, g(t) \bigr)_{H} \,\textup{d}s
\end{align*}
for~a.e.~$t \in (0,T)$, where $k'$ denotes the weak derivative of $k$. 
Since $k$ is nonnegative and nonincreasing, and $g(t) \in \partial \varphi^{t}(u(t))$ for~a.e.~$t \in (0,T)$, it follows from \eqref{eq:z-approx-1}, \eqref{eq:z-approx-2}, \eqref{eq:w-approx-1}, \eqref{eq:w-approx-2}  and the Cauchy--Schwarz inequality that
\begin{align} 
  & %%%%% [1行目]
  \bigl(\partial_{t}[k*(u-u_{0})](t),\, g(t)\bigr)_{H} \notag 
  \\
  &\quad \geq %%%%% [2,3,4,5行目] 
  \partial_{t}\bigl[k*\varphi^{ \bullet }(u( \bullet ))\bigr](t) 
  + k(t) \bigl( - \varphi^{t}(z_{t}) + \varphi^{0}(u_{0}) - \varphi^{0}(u_{0})\bigr) \notag 
  \\
  &\qquad -
  k(t)\|z_{t}-u_{0} \|_{H} \| g(t)\|_{H}  \notag 
  \\
  &\qquad 
  + \int_{0}^{t}  \bigl( -k'(s) \bigr)   \, \Bigl( -\varphi^{t}(w_{t}(t-s)) + \varphi^{t-s}(u(t-s))\Bigr)\,\textup{d}s \notag 
  \\
  &\qquad 
   -  \int_{0}^{t}  \bigl( -k'(s) \bigr)  \, \| w_{t}(t-s) - u(t-s) \|_{H} \| g(t) \|_{H} \,\textup{d}s  \notag 
  \\
  &\quad \geq %%%%% [6,7,8,9行目] 
  \partial_{t}\bigl[k*\varphi^{ \bullet }(u( \bullet ))\bigr](t) 
  - c_{1} t k(t)  \bigl(1 + |\varphi^{0}(u_{0})| \bigr) 
  - k(t) \varphi^{0}(u_{0}) 
  \notag 
  \\
  &\qquad -
   \bigl[c_{1} t k(t) \bigl(1 + |\varphi^{0}(u_{0})| \bigr) \bigr]^{1/2} 
  \, \bigl(c_{1} t k(t) \bigr)^{1/2} \| g(t)\|_{H}   \notag 
  \\
  &\qquad 
  - c_{2} \int_{0}^{t}  \bigl( -sk'(s) \bigr)  \Bigl( 1+ |\varphi^{t-s}(u(t-s))| \Bigr) \,\textup{d}s  \notag 
  \\
  &\qquad 
  - c_{2} \int_{0}^{t}  \bigl( -sk'(s) \bigr)  \Bigl( 1+ |\varphi^{t-s}(u(t-s))| \Bigr)^{1/2} \| g(t) \|_{H} \,\textup{d}s \label{eq:lem:proof:ineq1}
\end{align}
 for a.e. $t \in (0,T)$.  Since $k$ is nonnegative and nonincreasing, it holds that 
\begin{align} \label{eq:tk}
  0 \leq t k(t) 
  \;=\; \int_{0}^{t} k(t)\,\mathrm{d}s
  \;\leq\; \int_{0}^{t} k(s)\,\mathrm{d}s
  \;\leq\; \|k\|_{L^{1}(0,T)}
\end{align}
for a.e.~$t \in (0,T)$. 
Moreover, applying integration by parts and using the fact that $k$ is nonincreasing, we get  
\begin{align} 
    0 \leq \int_{0}^{t} \bigl(-s k'(s)\bigr)\,\mathrm{d}s
    &= \Bigl[ -s k(s)\Bigr]_{0}^{t}
       + \int_{0}^{t} k(s)\,\mathrm{d}s 
    =  -  t k(t) + \int_{0}^{t} k(s)\,\mathrm{d}s   \notag \\
    &\leq  \|k\|_{L^{1}(0,T)}
    \label{eq:skp}
\end{align}
for~all~$t \in (0,T]$. 
Combining \eqref{eq:lem:proof:ineq1}, \eqref{eq:tk}, \eqref{eq:skp} and Young's inequality, for every $\varepsilon \in (0,1)$, 
we obtain  
\begin{align*}
  &\bigl(\partial_{t}[k*(u-u_{0})](t),\, g(t)\bigr)_{H} \\ %%%%%%%
  &\quad \geq 
  \partial_{t}\bigl[k*\varphi^{\bullet}(u(\bullet))\bigr](t) 
  - c_{1} t k(t) \bigl(1 + |\varphi^{0}(u_{0})| \bigr) 
  - k(t)\varphi^{0}(u_{0}) 
  \\
  &\qquad 
  - \dfrac{1}{2\varepsilon}\, c_{1} t k(t) \bigl(1 + |\varphi^{0}(u_{0})| \bigr) 
  - \dfrac{\varepsilon}{2}\, c_{1} t k(t) \| g(t)\|_{H}^{2} 
  \\  
  &\qquad 
  - c_{2} \int_{0}^{t}  \bigl(-s k'(s)\bigr)\bigl(1+ |\varphi^{t-s}(u(t-s))|\bigr)\,\textup{d}s  
  \\ 
  &\qquad 
  - \dfrac{1}{2\varepsilon}\, c_{2} \int_{0}^{t}  \bigl(-s k'(s)\bigr)\bigl(1+ |\varphi^{t-s}(u(t-s))|\bigr)\,\textup{d}s 
  \\ 
  &\qquad 
  - \dfrac{\varepsilon}{2}\, c_{2} \int_{0}^{t} \bigl(-s k'(s)\bigr)\| g(t)\|_{H}^{2}\,\textup{d}s 
  \\[1ex]  
  &\quad =  %%%%%%%%
  \partial_{t}\bigl[k*\varphi^{\bullet}(u(\bullet))\bigr](t)
  - k(t)\varphi^{0}(u_{0}) 
  \\ &\qquad 
  - \dfrac{\varepsilon}{2}\left( c_{1} t k(t) + c_{2}\!\int_{0}^{t}\!\bigl(-s k'(s)\bigr)\,\mathrm{d}s \right)\| g(t)\|_{H}^{2}
  \\ &\qquad 
  - c_{1}\left(1 + \dfrac{1}{2\varepsilon}\right)t k(t) \bigl(1 + |\varphi^{0}(u_{0})| \bigr) 
  \\ &\qquad 
  - c_{2}\left(1 + \dfrac{1}{2\varepsilon}\right)\int_{0}^{t}  \bigl(-s k'(s)\bigr)\,\textup{d}s 
  \\ &\qquad 
  - c_{2}\left(1 + \dfrac{1}{2\varepsilon}\right)\int_{0}^{t}  \bigl(-s k'(s)\bigr)\,|\varphi^{t-s}(u(t-s))|\,\textup{d}s 
  \\[1ex]  
  &\quad \geq  %%%%%%%%%%%
  \partial_{t}\bigl[k*\varphi^{\bullet}(u(\bullet))\bigr](t)
  - k(t)\varphi^{0}(u_{0}) 
  - \dfrac{\varepsilon}{2}(c_{1}+c_{2})\,\|k\|_{L^{1}(0,T)}\| g(t)\|_{H}^{2} 
  \\ &\qquad 
  - c_{1}\left(1 + \dfrac{1}{2\varepsilon}\right)\|k\|_{L^{1}(0,T)} \bigl(1 + |\varphi^{0}(u_{0})| \bigr) 
  - c_{2}\left(1 + \dfrac{1}{2\varepsilon}\right)\|k\|_{L^{1}(0,T)}
  \\ &\qquad 
  - c_{2}\left(1 + \dfrac{1}{2\varepsilon}\right)\int_{0}^{t} \bigl(-s k'(s)\bigr)\,|\varphi^{t-s}(u(t-s))|\,\textup{d}s 
  \\&\quad \geq  %%%%%%%%%%%
  \partial_{t}\bigl[k*\varphi^{\bullet}(u(\bullet))\bigr](t)
  - k(t)\varphi^{0}(u_{0}) 
  - \dfrac{\varepsilon}{2} \, (c_{1}+c_{2})\,  \|k\|_{L^{1}(0,T)} \, \| g(t)\|_{H}^{2} 
  \\ &\qquad 
  - (c_{1}+c_{2}) \left(1 + \dfrac{1}{2\varepsilon}\right) \|k\|_{L^{1}(0,T)} \bigl(1 + |\varphi^{0}(u_{0})| \bigr) 
  \\ &\qquad 
  - c_{2} \left(1 + \dfrac{1}{2\varepsilon}\right) \int_{0}^{t} \bigl(-s k'(s)\bigr)\,|\varphi^{t-s}(u(t-s))|\,\textup{d}s 
\end{align*}
for a.e.\ $t \in (0,T)$. 
Thus we can take a constant $c \in (0,\infty)$ such that 
\eqref{eq:device-1} holds.  
%--------------------------------------------------------------
% Step 1. Estimate via Fubini--Tonelli
%--------------------------------------------------------------

Next, we prove \eqref{eq:device-2}. 
Using Fubini--Tonelli's theorem together with \eqref{eq:skp}, we obtain
\begin{align*}
  & %%%%% [1st line]
  \int_{0}^{t}\!\int_{0}^{s} 
      \bigl( -r k'(r) \bigr) \,
     \bigl|\varphi^{ s-r }(u(s-r))\bigr| \,\textup{d}r\,\textup{d}s \\
  &\quad =  %%%%% [2nd line]
  \int_{0}^{t}  \bigl( -r k'(r) \bigr) 
     \int_{r}^{t} \bigl|\varphi^{s-r}(u(s-r))\bigr|
     \,\textup{d}s\,\textup{d}r \\
  &\quad =  %%%%% [3rd line]
  \int_{0}^{t}  \bigl( -r k'(r) \bigr) 
     \int_{0}^{t-r} \bigl|\varphi^{ \sigma }(u(  \sigma ))\bigr|
     \,\textup{d}  \sigma  \,\textup{d}r 
  \leq    \| k\|_{L^{1}(0,T)} \int_{0}^{t} \left| \varphi^{\sigma} \bigl( u(\sigma) \bigr)\right| \, \textup{d} \sigma
\end{align*}
for~all~$t \in (0,T]$. 
Hence,  integrating  both sides of \eqref{eq:device-1} over $(0,t)$, for every $\varepsilon \in (0,1)$,  we obtain
\begin{align*}
  & %%%%% [1st line]
  \int_{0}^{t}
    \bigl(\partial_{ t }[k*(u-u_{0})](s),\, g(s)\bigr)_{H}
    \,\textup{d}s \\
  &\quad \geq  %%%%% [2nd–4th lines]
  \bigl[k*\varphi^{ \bullet }(u( \bullet ))\bigr](t)
  - \varphi^{0}(u_{0}) \int_{0}^{t} k(s)\,\textup{d}s 
  - \varepsilon\,  c  \, \|k\|_{L^{1}(0,T)} 
  \int_{0}^{t}\! \|g(s)\|_{H}^{2} \,\textup{d}s \\
  &\qquad
  - \dfrac{c}{\varepsilon} t \, \|k\|_{L^{1}(0,T)} \,  
     \bigl( 1+|\varphi^{0}(u_{0})| \bigr)  \\
  &\qquad
  - \dfrac{c}{\varepsilon} \int_{0}^{t}\!\int_{0}^{s}
       \bigl(-r k'(r)\bigr)  \,
      \bigl|\varphi^{s-r}(u(s-r))\bigr|\,\textup{d}r\,\textup{d}s \\
  &\quad \geq  %%%%% [5th–6th lines] 
  \bigl[k*\varphi^{ \bullet }(u( \bullet ))\bigr](t)
  - \varphi^{0}(u_{0}) \int_{0}^{t}k(s) \, \textup{d}s  
  - \varepsilon\,  c  \,\|k\|_{L^{1}(0,T)}
    \int_{0}^{t}\!\|g(s)\|_{H}^{2}\,\textup{d}s  \notag \\
  &\qquad
  - \dfrac{ c }{\varepsilon}  t  \,\|k\|_{L^{1}(0,T)}\bigl(1+|\varphi^{0}(u_{0})|\bigr)  
  - \dfrac{ c }{\varepsilon} \|k\|_{L^{1}(0,T)} 
    \int_{0}^{t} \bigl|\varphi^{s}(u(s))\bigr|\,\textup{d}s  
\end{align*}
for all~$t \in (0,T]$. 
 Thus \eqref{eq:device-2} has been proved. This completes the proof.  
\end{proof}

Here, combining \cref{lem:chainrule1} with the results stated in 
\cref{subsec:time-nonlocal-ops}, we can obtain the following lemma, 
which will play a crucial role in the proof of \cref{thm:main2}. 

\begin{lemma}[Nonlocal chain-rule formula for $(k,\ell) \in PC$] \label{lem:chainrule2}
For each $t \in [0,T]$, let $\varphi^{t} \colon H \to (-\infty,\infty]$ be a proper lower-semicontinuous convex functional, and suppose that  both  assumptions \textup{(A1)} and  \textup{(A2)}  are satisfied.  
Let $(k,\ell) \in PC$, $u_{0} \in D(\varphi^{0})$, and let $u,g \in L^{2}(0,T;H)$ be such that $u-u_{0} \in D(\mathcal{B})$  (see \cref{subsec:time-nonlocal-ops}), 
$\varphi^{ \bullet }(u( \bullet )) \in L^{1}(0,T)$, $u(t) \in D(\partial \varphi^{t})$ and 
$g(t) \in \partial \varphi^{t}(u(t))$ for~a.e.~$t \in (0,T)$. 
Then, for all $\varepsilon \in (0,1)$, the following inequality holds\/\textup{:}  
\begin{align} 
  & %%%%% [1st line]
  \int_{0}^{t}
    \bigl(\partial_{ t }[k*(u-u_{0})](s),\, g(s)\bigr)_{H}
    \,\textup{d}s  \notag \\
  &\quad \geq  %%%%% [2nd–4th lines]
  \bigl[k*\varphi^{ \bullet }(u( \bullet )) \bigr](t) - \varphi^{0}(u_{0}) \int_{0}^{t}k(s) \, \textup{d}s  
  - \varepsilon\,  c  \,\|k\|_{L^{1}(0,T)}
    \int_{0}^{t}\!\|g(s)\|_{H}^{2}\,\textup{d}s  \notag \\
  &\qquad  %%%%% [5th line]
  - \dfrac{ c }{\varepsilon}\,  t  \,\|k\|_{L^{1}(0,T)}\bigl(1+|\varphi^{0}(u_{0})|\bigr)  \notag \\
  &\qquad  %%%%% [6th line]
  - \dfrac{ c  }{\varepsilon} \, \|k\|_{L^{1}(0,T)}
    \int_{0}^{t} \bigl|\varphi^{s}(u(s))\bigr|\,\textup{d}s \label{eq:device-3}
\end{align}
for all $t \in (0,T]$, where  $c$  is the constant  of  \eqref{eq:device-2}  in  \cref{lem:chainrule1} and is independent of $u$, $g$, $k$ and $\varepsilon$. 
In particular, there exists a constant $C  \in [0,\infty)$ independent of $u$ and $g$ such that, for all $\varepsilon \in (0,1)$, it holds that   
\begin{align} 
  &\int_{0}^{t}
    \bigl(\partial_{ t  }[k*(u-u_{0})](s),\, g(s)\bigr)_{H}
    \,\textup{d}s  \notag \\
  &\quad \geq  %%%%% [2nd–4th lines]
  \bigl[k*\varphi^{ \bullet }(u( \bullet )) \bigr](t) - \varphi^{0}(u_{0}) \int_{0}^{t}k(s) \, \textup{d}s  
  - \varepsilon\,  C 
    \int_{0}^{t}\!\|g(s)\|_{H}^{2}\,\textup{d}s  \notag \\
  &\qquad  %%%%% [5th–6th lines]
  - \dfrac{ C }{\varepsilon} \left[ T \bigl(1 + |\varphi^{0}(u_{0})|\bigr) + \int_{0}^{t} \bigl|\varphi^{s}(u(s))\bigr| \,\textup{d}s \right]  
  \label{eq:device-4}
\end{align}
for all $t \in (0,T]$.  
\end{lemma}

\begin{proof}
For each $\lambda \in (0,\infty)$, let $k_{\lambda}$ be the solution to the Volterra equation \eqref{eq:Volterra-klambda} introduced in \cref{subsec:time-nonlocal-ops}.
As stated in \cref{subsec:time-nonlocal-ops}, $k_{\lambda}$ belongs to $W^{1,1}(0,T)$ and is nonnegative and nonincreasing.
Therefore, for each $\varepsilon \in (0,1)$,  it follows from \eqref{eq:device-2} that  
\begin{align} 
  & %%%%% [1st line]
  \int_{0}^{t}
    \bigl(\partial_{ t }[k_{\lambda}*(u-u_{0})](s),\, g(s)\bigr)_{H}
    \,\textup{d}s  \notag \\
  &\quad \geq  %%%%% [2nd–4th lines]
   \bigl[k_{\lambda}*\varphi^{\bullet}(u(\bullet)) \bigr](t)  - \varphi^{0}(u_{0}) \int_{0}^{t}k_{\lambda}(s) \, \textup{d}s  
  - \varepsilon\,  c  \,\|k_{\lambda}\|_{L^{1}(0,T)}
    \int_{0}^{t}\!\|g(s)\|_{H}^{2}\,\textup{d}s  \notag \\
  &\qquad  %%%%% [5th line]
  - \dfrac{ c }{\varepsilon} \,   t  \, \|k_{\lambda}\|_{L^{1}(0,T)} \, \bigl(1+|\varphi^{0}(u_{0})|\bigr)  \notag \\
  &\qquad  %%%%% [6th line]
  - \dfrac{  c  }{\varepsilon}  \, \|k_{\lambda}\|_{L^{1}(0,T)}  
    \int_{0}^{t} \bigl|\varphi^{s}(u(s))\bigr|\,\textup{d}s 
    \label{eq:chainrule2-approx} 
\end{align}
for all $t \in (0,T]$, where $c$ is the constant of \eqref{eq:device-2}  in  \cref{lem:chainrule1} and is independent of $u$, $g$, $k$ and $\varepsilon$.   
By  virtue of the properties of the Yosida approximation (see \cref{prop:JY}) together with the assumption $u-u_{0} \in D(\mathcal{B})$, we have 
\begin{align*}
\partial_{t}\bigl[k_{\lambda}*(u-u_{0})\bigr]
=
\mathcal{B}_{\lambda}(u-u_{0}) 
\to  
\mathcal{B}(u-u_{0}) 
=
\partial_{t}\bigl[k*(u-u_{0})\bigr]
\end{align*} 
in $L^{2}(0,T;H)$ as $\lambda \to 0_{+}$.  
Moreover, as stated in \cref{subsec:time-nonlocal-ops}, we have $k_{\lambda} \to k$ in $L^{1}(0,T)$ as $\lambda \to 0_{+}$.    
Hence, by letting $\lambda \to 0_{+}$ in \eqref{eq:chainrule2-approx}, we  obtain  \eqref{eq:device-3}. The proof is complete. 
\end{proof}

\section{Proof of \cref{thm:wellposedness}}  \label{Sec: Proof of well-posedness}

In this section, we  prove the  uniqueness and continuous dependence on initial data of strong solutions to \eqref{E:main-equation1}.    
Let $(u_{i},\xi_{i}) \in \eqref{E:main-equation1}_{u_{0,i},\,f_{i}}$ for $i=1,2$, where $u_{0,i} \in H$ and $f_{i} \in L^{2}(0,T;H)$.    
Then  we have  $k*\bigl( (u_{1}-u_{2}) - (u_{0,1}-u_{0,2}) \bigr) \in W^{1,2}(0,T;H)$ and  
  $\bigl[ k* \bigl( (u_{1}-u_{2}) - (u_{0,1}-u_{0,2}) \bigr) \bigr] (0) = 0$, that is, $(u_{1}-u_{2}) - (u_{0,1}-u_{0,2})  \in D(\mathcal{B})$ (see \cref{subsec:time-nonlocal-ops}). 
Furthermore, since $\xi_{1}(t) \in \partial \varphi^{t}\bigl(u_{1}(t)\bigr)$ and $\xi_{2}(t) \in \partial \varphi^{t}\bigl(u_{2}(t)\bigr)$ for a.e.\ $t \in (0,T)$, it follows from the monotonicity of $\partial \varphi^{t}$  (see \cref{subsec:subdifferential}) that   
\begin{align} \label{eq:monotonicity-subdiff}
  \int_{0}^{t} \bigl( \xi_{1}(s) - \xi_{2}(s),\, u_{1}(s) - u_{2}(s) \bigr)_{H} \, \textup{d}s \geq 0
\end{align}
for a.e.\ $t \in (0,T)$.
 Moreover, we observe that 
\begin{align*}
  \mathcal{B} \bigl[(u_{1}-u_{2}) - (u_{0,1}-u_{0,2}) \bigr](t) 
  = -\bigl(\xi_{1}(t) - \xi_{2}(t)\bigr) +  f_{1}(t) - f_{2}(t) 
\end{align*}
for a.e.\ $t \in (0,T)$.
By multiplying  both sides of the above identity by $u_{1}-u_{2}\in L^{2}(0,T;H)$ and integrating  it  over $(0,t)$, 
we deduce from \eqref{eq:monotonicity-subdiff}, \cref{P:frac_chain} and H\"older's inequality that 
\begin{align} 
  & \dfrac{1}{2} \bigl[k*\bigl( \|u_{1}-u_{2}\|_{H}^{2} - \|u_{0,1}-u_{0,2}\|_{H}^{2}\bigr)\bigr](t)  
  \notag  \\& \quad \leq 
  \int_{0}^{t} \bigl( \mathcal{B} \bigl[(u_{1}-u_{2}) - (u_{0,1}-u_{0,2}) \bigr](s) ,  (u_{1}-u_{2})(s) \bigr)_{H}  \, \textup{d}s 
  \notag  \\&\quad = 
  \int_{0}^{t}  -  \bigl(\xi_{1}(s) - \xi_{2}(s),\, u_{1}(s) - u_{2}(s) \bigr)_{H}\,\textup{d}s 
  \notag  \\&\qquad +   
  \int_{0}^{t}  \bigl(f_{1}(s) - f_{2}(s),\, u_{1}(s) - u_{2}(s) \bigr)_{H}\,\textup{d}s  
  \notag  \\&\quad \leq    
  \| f_{1}-f_{2}\|_{L^{2}(0,T;H)} \, \| u_{1}-u_{2}\|_{L^{2}(0,T;H)} 
   \label{eq:1-2-energy-ineq}
\end{align} 
for~a.e.~$t \in (0,T)$. 
Since $k * \ell \equiv 1$ on $(0,T)$ and $(1 * \ell)(t) \leq \|\ell\|_{L^{1}(0,T)}$ for a.e.\ $t \in (0,T)$,  
by convolving both sides of \eqref{eq:1-2-energy-ineq} with $\ell$ and using Young's inequality, we obtain  
\begin{align*} 
&\dfrac{1}{2}\int_{0}^{t}   \| u_{1}(s) - u_{2}(s) \|_{H}^{2}  \, \textup{d}s
 \\ & \quad = 
 \dfrac{1}{2}\int_{0}^{t} \left(  \| u_{1}(s) - u_{2}(s) \|_{H}^{2} - \| u_{0,1}-u_{0,2} \|_{H}^{2}  \right) \, \textup{d}s 
 + 
  \dfrac{t}{2}  \| u_{0,1}-u_{0,2} \|_{H}^{2}
\\ & \quad = 
\dfrac{1}{2} \Bigl[ \ell * \bigl[  k* \bigl( \|u_{1}-u_{2}\|_{H}^{2} - \|u_{0,1}-u_{0,2}\|_{H}^{2} \bigr) \bigr] \Bigr] (t)  
+ 
 \dfrac{t}{2}  \| u_{0,1}-u_{0,2} \|_{H}^{2}
\\ & \quad \leq
   \| \ell \|_{L^{1}(0,T)} \,   \| f_{1}-f_{2}\|_{L^{2}(0,T;H)} \, \| u_{1}-u_{2}\|_{L^{2}(0,T;H)} 
  +  \dfrac{t}{2}  \| u_{0,1}-u_{0,2} \|_{H}^{2}
  \notag  \\&\quad \leq  
  \|\ell\|_{L^{1}(0,T)}^{2} \, \| f_{1}-f_{2}\|_{L^{2}(0,T;H)}^{2} 
  + 
  \dfrac{1}{4} \, \| u_{1}-u_{2}\|_{L^{2}(0,T;H)}^{2} 
\\ & \qquad + 
 \dfrac{T}{2}  \| u_{0,1}-u_{0,2} \|_{H}^{2}
\end{align*}  
for a.e.\ $t \in (0,T)$. 
 Hence  it follows that  
\begin{align*} 
  & \| u_{1} - u_{2}\|_{L^{2}(0,T;H)}^{2}
  \\ & \quad \leq 
  4 \, \|\ell\|_{L^{1}(0,T)}^{2} \, \| f_{1}-f_{2}\|_{L^{2}(0,T;H)}^{2} 
  + 
   2  T \, \| u_{0,1}-u_{0,2} \|_{H}^{2} 
  \\ & \quad \leq 
  \, \max \bigl\{ 4  \, \|\ell\|_{L^{1}(0,T)}^{2},\,  2  T \bigr\}
  \, \Bigl( \| u_{0,1}-u_{0,2} \|_{H}^{2} + \| f_{1}-f_{2}\|_{L^{2}(0,T;H)}^{2} \Bigr).
\end{align*}
This completes the proof. \qed

%%%%%%%%%%%%%%%%%%%%%%%%%
%%%%%%%%%%%%%%%%%%%%%%%%%
%%%%%%%%%%%%%%%%%%%%%%%%%
%%%%%%%%%%%%%%%%%%%%%%%%%
%%%%%%%%%%%%%%%%%%%%%%%%%
%%%%%%%%%%%%%%%%%%%%%%%%%

\section{Proof of \cref{thm:main1}} \label{Sec: proof of existence of strong solutions for Sobolev spaces}

In this section, we provide  a  proof of \cref{thm:main1}. 
Let $\mathcal{A}$ and $\mathcal{B}$ be defined as in \cref{subsec:time-nonlocal-ops}, 
and let $\Phi$ denote the proper lower-semicontinuous convex functional on $L^{2}(0,T;H)$ defined as in  \cref{prop:Phi-t-properties}.

\subsection{Approximate problem} 
For $\nu \in (0,\infty)$ and $\lambda \in (0,1)$, 
we show that the following approximate problem admits a unique strong solution:  
\begin{equation} \label{E:main-equation1-approx}
  (\nu \mathcal{A} + \mathcal{B})(u_{\nu,\lambda} - u_{0}) + \partial \Phi_{\lambda}(u_{\nu,\lambda}) = f 
  \quad \textup{in } L^{2}(0,T;H).
\end{equation} 
It suffices to show that there exists a unique function 
$u_{\nu,\lambda} \in W^{1,2}(0,T;H)$ with $u_{\nu,\lambda}(0) = u_{0}$ such that 
\begin{equation} 
  \nu \, \partial_{t}u_{\nu, \lambda}(t) 
  +  \partial_{t}  \bigl[k*(u_{\nu,\lambda}-u_{0})\bigr](t) 
  + \partial \varphi^{t}_{\lambda} \bigl(u_{\nu,\lambda}(t)\bigr) 
  = f(t)
  \quad \textup{ in $H$ } 
\end{equation}
for a.e.\ $t \in (0,T)$. 
As already mentioned  in \cref{subsec:subdifferential,subsec:Time-dependent subdifferential}, the mapping $(t,w) \mapsto \partial \varphi^{t}_{\lambda}(w)$ is a Carath\'eodory function,  that is, 
for each $w \in H$, the map $t \mapsto \partial \varphi^{t}_{\lambda}(w)$ is strongly measurable in $(0,T)$, 
and moreover, for each $t \in [0,T]$, the map $w \mapsto \partial \varphi^{t}_{\lambda}(w)$ is continuous  in $H$ 
(see \cref{prop:JY,prop:MY,prop:Phi}). 
The following lemma guarantees that the approximate problem \eqref{E:main-equation1-approx} admits a unique strong solution. %(see \cref{prop:JY,prop:MY,prop:time-subdiff,prop:Phi,prop:Phi-t-properties}). 

\begin{lemma}\label{lem:existence-v}
 Let $X$ be a real Banach space.  
Let $k \in L^{1}(0,T)$,  $p \in [1,\infty]$, and   
let $F \colon  [0,T]  \times X \to X$ be a Carath\'eodory function such that there exist a constant $C_{0} \in [0,\infty)$ and a function $\rho \in L^{p}(0,T)$ satisfying the following\/\textup{:} 
\begin{align}
  \|F(t,w) \|_{X} &\leq C_{0} \| w\|_{X} + |\rho(t)|, 
  \label{eq:F-growth}
  \\ 
  \|F(t,x) - F(t,y)\|_{X} &\leq C_{0} \|x-y\|_{X} 
  \label{eq:F-Lip}
\end{align}
for all $w,x,y \in X$ and for a.e.\ $t \in (0,T)$. 
Then, for each $\nu \in (0,\infty)$, $v_{0} \in X$ and  $f \in L^{p}(0,T;X)$, there exists a unique function $v \in W^{1,p}(0,T;X)$ with $v(0)=v_{0}$ such that 
\begin{equation*} %\label{eq:diff-form}
  \nu \, \partial_{t}v(t) + \partial_{t}\bigl[k*(v-v_{0})\bigr](t) + F(t,v(t)) = f(t) 
  \quad \textup{in } X \quad \textup{for $t \in (0,T)$.} 
\end{equation*}
\end{lemma}

\begin{proof}
Let $\nu \in (0,\infty)$, $v_{0} \in X$ and let $f \in L^{p}(0,T;X)$ be fixed. 
By the fundamental theorem of calculus, it suffices to show that there exists a unique function 
$v \in W^{1,p}(0,T;X)$ with $v(0)=v_{0}$ satisfying
\begin{align} 
    &\nu\,[v(t)-v_{0}] + [k*(v-v_{0})](t) + \int_{0}^{t} F(s,v(s)) \,\textup{d}s
    \notag  \\  
    &\quad = 
    \int_{0}^{t} f(s) \,\textup{d}s 
    \quad \textup{in $X$ }
    \label{eq:int-form} 
\end{align}
for all $t \in (0,T]$. 
Therefore, it suffices to show that there exists a unique function 
$v \in \Theta_{v_{0}} := \{ w \in W^{1,p}(0,T;X) \colon w(0) = v_{0}\}$ such that $v = \Lambda(v)$, 
where $\Lambda \colon \Theta_{v_{0}} \to \Theta_{v_{0}}$ is defined by 
\begin{align*} 
[ \Lambda (w)](\bullet) 
&:= 
v_{0} 
- \dfrac{1}{\nu} [k*(w-v_{0})](\bullet) 
- \dfrac{1}{\nu} \int_{0}^{\bullet} F(s,w(s)) \,\textup{d}s 
+ \dfrac{1}{\nu} \int_{0}^{\bullet}  f(s) \,\textup{d}s 
\\
&=  
v_{0} 
- \dfrac{1}{\nu} 
\Bigl(   
    [k*(w-v_{0})](\bullet) 
  + \bigl[1*F( \bullet ,w( \bullet ))\bigr](\bullet)  
  - (1*f)(\bullet) 
\Bigr)
\end{align*}
for $w \in \Theta_{v_{0}}$. 
Due to \eqref{eq:F-growth}, we note that $F( \bullet ,w( \bullet )) \in L^{p}(0,T;X)$ for every $w \in L^{p}(0,T;X)$, and hence, the mapping $\Lambda$ is well-defined.

For each $c \in (0,\infty)$, a map $d_{c} \colon \Theta_{v_{0}} \times \Theta_{v_{0}} \to [0,\infty)$ is defined by  
\begin{align*}
  d_{ c }(u,v)
  &:=
  \bigl\|(u-v)( \bullet )\,\exp(-c \, \bullet)\bigr\|_{L^{p}(0,T;X)}
  \\ &\quad +
  \bigl\|(u-v)' (\bullet)\,\exp(-c \, \bullet) \bigr\|_{L^{p}(0,T;X)}
\end{align*} 
for $u,v \in \Theta_{v_{0}}$, where $w'$ denotes the weak derivative of $w$ for each $w \in W^{1,p}(0,T;X)$. 
Then $(\Theta_{v_{0}}, d_{c})$ is a complete metric space for all $c \in (0,\infty)$.  
We now claim that there exist $c_{0} \in (0,\infty)$ and $\kappa_{0} \in (0,1)$ such that
\begin{equation*}
   d_{c_{0}} \bigl( \Lambda (u), \Lambda (v)\bigr) 
  \leq \kappa_{0}\, d_{c_{0}} (u,v) 
\end{equation*}
for all $u,v \in \Theta_{v_{0}}$, that is, $\Lambda$ is a contraction mapping on $(\Theta_{v_{0}}, d_{c_{0}})$. 
Indeed, fix $c \in (0,\infty)$ and take arbitrary $u_{1},u_{2} \in \Theta_{v_{0}}$. 
Since  $\partial_{t} \bigl[k*(u_{1}-u_{2})\bigr](t) = [k*(u_{1}-u_{2})'](t)$, $\partial_{t} \bigl[1* \bigl( F \bigl(\bullet ,u_{1}(\bullet) \bigr) - F \bigl( \bullet ,u_{2}(\bullet) \bigr) \bigr)\bigr](t) = F \bigl(t ,u_{1}(t) \bigr) - F \bigl( t ,u_{2}(t) \bigr)$ and 
\begin{align*} 
  \| u_{1}(t) - u_{2}(t) \|_{X} = \bigl\| [1*(u_{1}-u_{2})'](t)\bigr\|_{X} 
  \leq \bigl[ 1 * \| (u_{1} - u_{2})'\|_{X} \bigr](t) 
\end{align*} 
for~a.e.~$t \in (0,T)$, 
it follows from \eqref{eq:F-Lip} and Minkowski's inequality that 
\begin{align}  
  %%% 1 line  
  &d_{ {c}} \bigl( \Lambda (u_{1}), \Lambda (u_{2}) \bigr)
  \notag \\ 
  &\quad =  %%%%2-5 line  
  \biggl\| 
    \dfrac{1}{\nu}   
    [k*(u_{1} - u_{2})](\bullet) \exp(-c \, \bullet) 
  \notag \\& \qquad \qquad + 
    \dfrac{1}{\nu} \Bigl[1* \bigl[ F \bigl( \bullet,u_{1}(\bullet) \bigr) - F\bigl(\bullet,u_{2}(\bullet) \bigr) \bigr]
    \Bigr](\bullet) \exp(-c \, \bullet)
    \biggr\|_{L^{p}(0,T;X)} 
  \notag \\ 
  &\qquad + 
  \biggl\| 
    \dfrac{1}{\nu}   
    [k*(u_{1} - u_{2})'](\bullet) \exp(-c \, \bullet) 
  \notag \\& \qquad \qquad + 
    \dfrac{1}{\nu} \Bigl( F \bigl( \bullet,u_{1}(\bullet) \bigr) - F\bigl(\bullet,u_{2}(\bullet) \bigr) \Bigr) \exp(-c \, \bullet)
    \biggr\|_{L^{p}(0,T;X)} 
  \notag \\ 
  &\quad \leq %%%%6-9 line  
  \dfrac{1}{\nu} \bigl\| [k*(u_{1}-u_{2})](\bullet) \exp(-c \, \bullet) \bigr\|_{L^{p}(0,T;X)}
  \notag \\ 
  &\qquad +
  \dfrac{1}{\nu} \bigl\| [k*(u_{1}-u_{2})'](\bullet) \exp(-c \, \bullet) \bigr\|_{L^{p}(0,T;X)}
  \notag \\ 
  &\qquad + 
    \dfrac{C_{0}}{\nu}
    \left\| 
    \bigl[ 1* \| u_{1}-u_{2}\|_{X} \bigr](\bullet) \exp(-c \, \bullet)  
     \right\|_{L^{p}(0,T)} 
  \notag \\ 
  &\qquad + 
    \dfrac{C_{0}}{\nu} 
    \left\| 
    \bigl[ 1* \|(u_{1}-u_{2})'\|_{X}  \bigr] (\bullet) \exp(-c \, \bullet)  
     \right\|_{L^{p}(0,T)}  
  \notag \\ 
  &\quad = %%%%10-13 line  
  \dfrac{1}{\nu} \left\| \bigl( k(\bullet) \exp(-c \, \bullet) \bigr)* \bigl( (u_{1}-u_{2})(\bullet) \exp(-c \, \bullet) \bigr) \right\|_{L^{p}(0,T;X)}
  \notag \\ 
  &\qquad +
  \dfrac{1}{\nu} \bigl\| \bigl( k(\bullet) \exp(-c \, \bullet) \bigr)* \bigl( (u_{1}-u_{2})'(\bullet) \exp(-c \, \bullet) \bigr) \bigr\|_{L^{p}(0,T;X)}
  \notag \\ 
  &\qquad + 
    \dfrac{C_{0}}{\nu}
    \left\| 
    \exp(-c \, \bullet) * \bigl\| (u_{1}-u_{2}) (\bullet) \exp(-c \, \bullet ) \bigr\|_{X} 
     \right\|_{L^{p}(0,T)} 
  \notag \\ 
  &\qquad + 
    \dfrac{C_{0}}{\nu} 
    \left\| 
    \exp(-c \, \bullet)* \bigl\| (u_{1}-u_{2})' (\bullet) \exp(-c \, \bullet )  \bigr\|_{X}  
     \right\|_{L^{p}(0,T)}.  
  \label{eq:dTheta-expansion}
\end{align}  
Here we used that, for every $g \in L^{1}(0,T)$ and $w \in L^{p}(0,T;X)$, the following identity holds:   
\begin{align} 
  &(g*w)(t) \exp(- c t) 
  = 
  \int_{0}^{t} g(t-s) \exp\bigl(- c(t-s)\bigr) \exp(- c s) w(s) \, \textup{d}s 
  \notag \\& \quad =
  \bigl[ \bigl( g(\bullet) \exp(- c \, \bullet) \bigr) * \bigl(w(\bullet) \exp(-c \, \bullet)\bigr) \bigr](t)  \label{eq:conv-exp-identity} 
\end{align} 
for a.e.\ $t \in (0,T)$. 
Applying Young's convolution inequality to \eqref{eq:dTheta-expansion}, we deduce that 
\begin{align*} 
%1 line 
& d_{c} \bigl( \Lambda(u_{1}), \Lambda(u_{2}) \bigr)
  \\ 
  &\quad \leq  %%%%2-5 line  
  \dfrac{1}{\nu} \left\| k(\bullet) \exp(-c \, \bullet) \right\|_{L^{1}(0,T)} 
  \left\| (u_{1}-u_{2})(\bullet) \exp(-c \, \bullet) \right\|_{L^{p}(0,T;X)}
  \notag \\ 
  &\qquad +
  \dfrac{1}{\nu} \left\| k(\bullet) \exp(-c \, \bullet) \right\|_{L^{1}(0,T)} 
  \left\| (u_{1}-u_{2})'(\bullet) \exp(-c \, \bullet) \right\|_{L^{p}(0,T;X)}
  \notag \\ 
  &\qquad + 
    \dfrac{C_{0}}{\nu}
    \left\|  \exp(-c \, \bullet) \right\|_{L^{1}(0,T)} 
    \left\| (u_{1}-u_{2}) (\bullet) \exp(-c \, \bullet) \right\|_{L^{p}(0,T;X)} 
  \notag \\ 
  &\qquad + 
    \dfrac{C_{0}}{\nu} 
    \left\|  \exp(-c \, \bullet) \right\|_{L^{1}(0,T)} 
    \left\| (u_{1}-u_{2})'(\bullet) \exp(-c \, \bullet) \right\|_{L^{p}(0,T;X)}  
  \\ 
  &\quad =  %%%%6 line  
  \left( 
      \dfrac{1}{\nu} \left\| k(\bullet) \exp(-c \, \bullet) \right\|_{L^{1}(0,T)} 
      + 
      \dfrac{C_{0}}{\nu} \left\|  \exp(-c \, \bullet) \right\|_{L^{1}(0,T)}  
  \right) 
  d_{c}(u_{1},u_{2}). 
\end{align*}
We can choose $c_{0} \in (0,\infty)$ sufficiently large such that  
\begin{equation*}
  \kappa_{0}
  :=  
  \dfrac{1}{\nu} \|k(\bullet) \exp(-c_{0} \, \bullet)\|_{L^{1}(0,T)} 
  + \dfrac{C_{0}}{\nu} \|\exp(-c_{0} \, \bullet)\|_{L^{1}(0,T)} \in [0,1). 
\end{equation*}
Thus, for all $u_{1},u_{2} \in \Theta_{v_{0}}$, we obtain 
\begin{equation*}
  d_{c_{0}}\bigl(\Lambda (u_{1}), \Lambda(u_{2})\bigr)  
  \leq \kappa_{0}\, d_{c_{0}} (u_{1},u_{2}),
\end{equation*}
which shows that $\Lambda$ is a contraction mapping on $(\Theta_{v_{0}}, d_{c_{0}})$. 
By virtue of Banach's fixed point theorem, there exists a unique function  
$v \in \Theta_{v_{0}}$ such that $\Lambda(v)=v$.   
This completes the proof. 
\end{proof}

%%%%%%%%%
%%%%%%%%%
%%%%%%%%%
%%%%%%%%%

\subsection{A priori estimate}

We next establish  a~priori estimates.  For each $\nu \in (0,1)$, let 
$u_{\nu} \in W^{1,2}(0,T;H)$  be such that  $u_{\nu}(0)=u_{0}$ \textup{(}i.e., $u_{\nu}-u_{0}\in D(\mathcal{A})=D(\nu\mathcal{A}+\mathcal{B})$,  see  \cref{subsec:time-nonlocal-ops}\textup{)}  and  
\begin{equation}\label{eq:star-nu}
  (\nu\mathcal{A}+\mathcal{B})\bigl(u_{\nu}-u_{0}\bigr) + \partial \Phi_{\nu}(u_{\nu}) = f 
  \quad \textup{in } L^{2}(0,T;H).
\end{equation} 
Testing \eqref{eq:star-nu} by $\mathcal{A}(u_{\nu}-u_{0})=\partial_{t}u_{\nu} = \partial_{t} (u_{\nu}-u_{0})$, we get 
\begin{align*}
  &\nu \bigl\|\mathcal{A}(u_{\nu}-u_{0})(s)\bigr\|_{H}^{2}
   + \bigl(\mathcal{B}(u_{\nu}-u_{0})(s),\, \mathcal{A}(u_{\nu}-u_{0})(s) \bigr)_{H}
  \\ &\qquad +
  \bigl(\partial \varphi^{s}_{\nu}(u_{\nu}(s)),\, \partial_{s}u_{\nu}(s)\bigr)_{H}
  \\
  &\quad = \bigl(f(s),\, \partial_{s}(u_{\nu}-u_{0})(s)\bigr)_{H}
  \\
  &\quad = \partial_{s}\bigl(f(s),\, u_{\nu}(s)-u_{0}\bigr)_{H}
          - \bigl(\partial_{s}f(s),\, u_{\nu}(s)-u_{0}\bigr)_{H} 
\end{align*}
for a.e.\ $s \in (0,T)$. 
Integrating both sides over $(0,t)$ and using \eqref{E:AB-energy-estimate}, \eqref{eq:star-nu}, \cref{prop:MY,prop:time-subdiff-energy} together with Young's inequality,  
there exist nonnegative functions $\eta_{1},\eta_{2} \in L^{1}(0,T)$ independent of $\nu$ such that
\begin{align*}
  & % [1 line]
  \int_{0}^{t} \nu \bigl\|\mathcal{A}(u_{\nu}-u_{0})(s)\bigr\|_{H}^{2}\,\textup{d}s 
  + \frac{1}{2}\bigl[\ell * \|\mathcal{B}(u_{\nu}-u_{0})\|_{H}^{2}\bigr](t)
  + \varphi_{\nu}^{t}\bigl(u_{\nu}(t)\bigr)
  \\
  &\quad \leq % [2–3 line]
  \varphi_{\nu}^{0}(u_{0})
  + \int_{0}^{t} \bigl\|\partial\varphi_{\nu}^{s}(u_{\nu}(s))\bigr\|_{H}^{2}\,\textup{d}s
  + \int_{0}^{t} \eta_{1}(s)\bigl|\varphi_{\nu}^{s}(u_{\nu}(s))\bigr|\,\textup{d}s
  + \int_{0}^{t} \eta_{2}(s)\,\textup{d}s
  \\
  &\quad \quad
  +  \bigl(f(t),\,u_{\nu}(t)-u_{0}\bigr)_{H} 
  -  \int_{0}^{t} \bigl(\partial_{s}f(s),\,u_{\nu}(s)-u_{0}\bigr)_{H}\,\textup{d}s
  \\
  &\quad = % [4–6 line]
  \varphi_{\nu}^{0}(u_{0})
  + \int_{0}^{t} \eta_{2}(s)\,\textup{d}s
  + \bigl(f(t),\,u_{\nu}(t)-u_{0}\bigr)_{H}
  \\
  &\quad \quad
  + \int_{0}^{t} \bigl\|-(\nu\mathcal{A}+\mathcal{B})(u_{\nu}-u_{0})(s) + f(s)\bigr\|_{H}^{2}\,\textup{d}s
  + \int_{0}^{t} \eta_{1}(s)\bigl|\varphi_{\nu}^{s}(u_{\nu}(s))\bigr|\,\textup{d}s
  \\
  &\quad \quad
  - \int_{0}^{t} \bigl(\partial_{s}f(s),\,u_{\nu}(s)-u_{0}\bigr)_{H}\,\textup{d}s
  \\
  &\quad \leq % [7–11 line]
  \varphi^{0}(u_{0})
  + \int_{0}^{t} \eta_{2}(s)\,\textup{d}s
  + \|f\|_{L^{\infty}(0,T;H)} \,\|u_{\nu}(t)-u_{0}\|_{H}
  \\
  &\quad \quad
  + 4 \int_{0}^{t} \|f(s)\|_{H}^{2}\,\textup{d}s
  + 4 \nu^{2} \int_{0}^{t} \bigl\|\mathcal{A}(u_{\nu}-u_{0})(s)\bigr\|_{H}^{2}\,\textup{d}s
  \\
  &\quad \quad
  + 4 \int_{0}^{t} \bigl\|\mathcal{B}(u_{\nu}-u_{0})(s)\bigr\|_{H}^{2}\,\textup{d}s
  + \int_{0}^{t} \eta_{1}(s)\bigl|\varphi_{\nu}^{s}(u_{\nu}(s))\bigr|\,\textup{d}s
  \\
  &\quad \quad
  + \frac{1}{2}\int_{0}^{t} \|\partial_{s}f(s)\|_{H}^{2}\,\textup{d}s
  + \frac{1}{2}\int_{0}^{t} \|u_{\nu}(s)-u_{0}\|_{H}^{2}\,\textup{d}s
\end{align*}
for a.e.\ $t \in (0,T)$. 
Here we used the following elementary inequality: 
\begin{align*} 
\|a+b+c \|_{H}^{2} 
&\leq 
\bigl(\|a\|_{H} + \|b\|_{H} + \|c\|_{H}\bigr)^{2} 
\leq    
4\bigl(\|a\|_{H}^{2} + \|b\|_{H}^{2} + \|c\|_{H}^{2}\bigr)
\end{align*} 
for all $a,b,c \in H$. 
Consequently, there exists a constant $C_{0} \in [0,\infty)$  independent of $\nu$ such that
\begin{align}
  & % [1 line]
  (\nu-4\nu^{2}) \int_{0}^{t} \bigl\|\mathcal{A}(u_{\nu}-u_{0})(s)\bigr\|_{H}^{2}\,\textup{d}s
  + \frac{1}{2}\bigl[\ell * \|\mathcal{B}(u_{\nu}-u_{0})\|_{H}^{2}\bigr](t)
  \notag \\
  &\quad \quad
  + \varphi_{\nu}^{t}\bigl(u_{\nu}(t)\bigr)
  - \|f\|_{L^{\infty}(0,T;H)}\,\|u_{\nu}(t)-u_{0}\|_{H}
  \notag \\
  &\quad \leq % [2–4 line]
  C_{0}
  + 4 \int_{0}^{t} \bigl\|\mathcal{B}(u_{\nu}-u_{0})(s)\bigr\|_{H}^{2}\,\textup{d}s
  + \int_{0}^{t} \eta_{1}(s)\bigl|\varphi_{\nu}^{s}(u_{\nu}(s))\bigr|\,\textup{d}s
  \notag \\
  &\qquad 
  + \frac{1}{2}\int_{0}^{t} \|u_{\nu}(s)-u_{0}\|_{H}^{2}\,\textup{d}s  \label{eq:proof-step1}
\end{align}
for a.e. $t\in(0,T)$. 
 Since $(k,\ell) \in PC$ and  $u_{\nu}-u_{0}\in D(\mathcal{A})\subset D(\mathcal{B})$, 
 it follows from H\"{o}lder's inequality that   
\begin{align}
  & %%%%% [1行目]
  \|u_{\nu}(t)-u_{0}\|_{H}^{2} 
  = 
  \bigl\|[\ell* \mathcal{B}(u_{\nu}-u_{0})](t) \bigr\|_{H}^{2} 
  \notag \\
  &= \left\|\int_{0}^{t} \ell^{\frac{1}{2}}(t-s)\,\ell^{\frac{1}{2}}(t-s)\,
      \mathcal{B}(u_{\nu}-u_{0})(s)\,\mathrm{d}s\right\|_{H}^{2} 
  \notag \\[0.25em]
  &\leq %%%%% [2行目, Hölder]
  \left(\int_{0}^{t}\ell(t-s)\,\mathrm{d}s\right)
        \left(\int_{0}^{t}\ell(t-s)\,
        \bigl\|\mathcal{B}(u_{\nu}-u_{0})(s)\bigr\|_{H}^{2}\,\mathrm{d}s\right) \notag \\[0.25em]
  &= %%%%% [2行目, Hölder]
  \left(\int_{0}^{t}\ell(s)\,\mathrm{d}s\right)
        \bigl[\ell * \|\mathcal{B}(u_{\nu}-u_{0})\|_{H}^{2}\bigr](t) \notag \\[0.25em]
  &\leq %%%%% [3行目]
  \bigl(\|\ell\|_{L^{1}(0,T)}+1\bigr)\,
  \bigl[\ell * \|\mathcal{B}(u_{\nu}-u_{0})\|_{H}^{2}\bigr](t) \label{eq:u-nu-u0-energy-bound}
\end{align}
for a.e. $t \in (0,T)$. 
Here we used the fact that, for each $v \in D(\mathcal{B})$, the following identity holds: 
\begin{align}
  v(t)
  &= \partial_t\!\left(\int_0^t v(s)\,\mathrm{d}s\right)
   = \partial_t (1 * v)(t)  \notag \\
  &= \partial_t\! \bigl[ \bigl( ( \ell * k ) * v \bigr)\bigr](t)
   = \bigl[\ell * \mathcal{B}(v)\bigr](t) \label{eq:B-repr}
\end{align} 
for a.e. $t \in (0,T)$. 
 From \eqref{eq:td-subdiff-varphi-lambda-bound}, \eqref{eq:u-nu-u0-energy-bound} %, \cref{prop:time-subdiff} 
 and Young's inequality, for every $\varepsilon \in (0,1)$, we observe that   
\begin{align*}
  &\frac{1}{2}\bigl[\ell * \|\mathcal{B}(u_{\nu}-u_{0})\|_{H}^{2}\bigr](t)
  + \varphi_{\nu}^{t}\bigl(u_{\nu}(t)\bigr)
  - \|f\|_{L^{\infty}(0,T;H)}\,\|u_{\nu}(t)-u_{0}\|_{H} % [1行目] 
  \notag \\ 
  &\quad = 
  \frac{1}{4}\bigl[\ell * \|\mathcal{B}(u_{\nu}-u_{0})\|_{H}^{2}\bigr](t)
  + \frac{1}{4}\bigl[\ell * \|\mathcal{B}(u_{\nu}-u_{0})\|_{H}^{2}\bigr](t) 
  \notag \\ %[2行目]
  &\qquad + 
  \varphi_{\nu}^{t}\bigl(u_{\nu}(t)\bigr)
  - \|f\|_{L^{\infty}(0,T;H)}\,\|u_{\nu}(t)-u_{0}\|_{H} % [3行目]
  \notag \\[0.25em]
  &\quad \geq %%%% [4行目]
  \frac{1}{4}\bigl[\ell * \|\mathcal{B}(u_{\nu}-u_{0})\|_{H}^{2}\bigr](t)
  + \frac{1}{4}\bigl(1+\|\ell\|_{L^{1}(0,T)}\bigr)^{-1}\|u_{\nu}(t)-u_{0}\|_{H}^{2}
  \notag \\[0.25em]
  &\qquad %%%% [5行目]
  + \bigl|\varphi_{\nu}^{t}(u_{\nu}(t))\bigr|
  -  D  \|u_{\nu}(t)\|_{H} -  D    
  - \|f\|_{L^{\infty}(0,T;H)}\,\|u_{\nu}(t)-u_{0}\|_{H}
  \notag \\ 
  &\quad \geq %%%% [6行目]
  \frac{1}{4}\bigl[\ell * \|\mathcal{B}(u_{\nu}-u_{0})\|_{H}^{2}\bigr](t)
  + \frac{1}{4}\bigl(1+\|\ell\|_{L^{1}(0,T)}\bigr)^{-1}\|u_{\nu}(t)-u_{0}\|_{H}^{2}
  \notag \\[0.25em]
  &\qquad %%%% [7行目]
  + \bigl|\varphi_{\nu}^{t}(u_{\nu}(t))\bigr|
  -  D  \bigl( \|u_{\nu}(t)-u_{0}\|_{H} + \|u_{0}\|_{H}\bigr) -  D  
  \notag \\[0.25em]
  &\qquad %%%% [8行目]
  -  \dfrac{\varepsilon}{2}  \|u_{\nu}(t)-u_{0}\|_{H}^{2} -  \dfrac{1}{2 \varepsilon}   \|f\|_{L^{\infty}(0,T;H)}^{2}
  \notag \\[0.25em]
  &\quad \geq %%%% [5行目]
  \frac{1}{4} \bigl[\ell * \|\mathcal{B}(u_{\nu}-u_{0})\|_{H}^{2}\bigr](t)
  +  \frac{1}{4}  \bigl(1+\|\ell\|_{L^{1}(0,T)}\bigr)^{-1}\|u_{\nu}(t)-u_{0}\|_{H}^{2}
  \notag \\[0.25em]
  &\quad \quad %%%% [6行目]
  + \bigl|\varphi_{\nu}^{t}(u_{\nu}(t))\bigr|
  -  \dfrac{\varepsilon}{2}   \|u_{\nu}(t)-u_{0}\|_{H}^{2}  - \dfrac{1}{2 \varepsilon} D^{2}  -  D  \| u_{0}\|_{H}  -  D  
  \notag \\[0.25em]
  &\qquad - %%%% [7行目]
    \dfrac{\varepsilon}{2}  \|u_{\nu}(t)-u_{0}\|_{H}^{2} 
  - \dfrac{1}{2 \varepsilon}  \|f\|_{L^{\infty}(0,T;H)}^{2}  
  %\label{eq:energy-lower-bound}
\end{align*}
for~a.e.~$t\in(0,T)$,  where $D$ is a constant satisfying \eqref{eq:td-subdiff-varphi-lambda-bound}.  
Hence there exists a constant $\varepsilon_{0} \in (0,1/4)$ independent of $\nu$ such that 
\begin{align} 
  &\frac{1}{2}\bigl[\ell * \|\mathcal{B}(u_{\nu}-u_{0})\|_{H}^{2}\bigr](t)
  + \varphi_{\nu}^{t}\bigl(u_{\nu}(t)\bigr)
  - \|f\|_{L^{\infty}(0,T;H)}\,\|u_{\nu}(t)-u_{0}\|_{H} % [1行目] 
  \notag \\ & \quad \geq 
  \dfrac{1}{4}       
  \bigl[\ell * \|\mathcal{B}(u_{\nu}-u_{0})\|_{H}^{2}\bigr](t)
  + 
  \varepsilon_{0}
  \|u_{\nu}(t)-u_{0}\|_{H}^{2}
  \notag \\ & \qquad  + 
  \bigl|\varphi_{\nu}^{t}(u_{\nu}(t))\bigr| 
  - \dfrac{1}{\varepsilon_{0}} 
  \label{eq:energy-lower-bound} 
\end{align}
for~a.e.~$t \in (0,T)$. 
Using \eqref{eq:proof-step1} and \eqref{eq:energy-lower-bound} together with $k * \ell \equiv 1$ on $(0,T)$, 
there exist  a constant  $C_{3} \in (0,\infty)$ and a nonnegative function $\eta_{3}\in L^{1}(0,T)$ independent of $\nu$ such that
\begin{align}
  & % [1 line]
  (\nu-4\nu^{2})\int_{0}^{t}
    \bigl\|\mathcal{A}(u_{\nu}-u_{0})(s)\bigr\|_{H}^{2}\,\textup{d}s
    +  \dfrac{1}{C_{3}}  F_{\nu}(t)
  \notag \\
  &\quad  \leq % [2 line]       
  C_{3} + 
  C_{3} \int_{0}^{t} \bigl\|\mathcal{B}(u_{\nu}-u_{0})(s)\bigr\|_{H}^{2}\,\textup{d}s 
  + 
  \int_{0}^{t}\eta_{3}(s)\,F_{\nu}(s)\,\textup{d}s  
  \notag \\ 
  &\quad = % [3 line]       
  C_{3} + 
  C_{3} \bigl[ (k * \ell) * \|\mathcal{B}(u_{\nu}-u_{0}) \|_{H}^{2}\bigr](t) 
  + 
  \int_{0}^{t}\eta_{3}(s)\,F_{\nu}(s)\,\textup{d}s  
  \notag \\ 
  &\quad \leq % [4 line]
    C_{3}
    + C_{3} \bigl(k * F_{\nu}\bigr)(t)
    + \int_{0}^{t}\eta_{3}(s)\,F_{\nu}(s)\,\textup{d}s
  \label{eq:energy-ineq-Fnu}
\end{align}
for a.e.\ $t\in(0,T)$, where 
\begin{align*} 
  F_{\nu}(  \bullet )
  := \|u_{\nu}(  \bullet  )-u_{0}\|_{H}^{2}
     + \bigl|\varphi^{ \bullet }_{\nu}(u_{\nu}( \bullet ))\bigr|
     + \bigl[\ell * \|\mathcal{B}(u_{\nu}-u_{0})\|_{H}^{2}\bigr]( \bullet ). 
\end{align*}
From the embedding $u_{\nu} \in W^{1,2}(0,T;H) \subset L^{\infty}(0,T;H)$ 
together with \eqref{eq:td-subdiff-varphi-lambda-bound}, \eqref{eq:AB-Holder} and \cref{prop:JY,prop:MY,prop:time-subdiff}, we can deduce that $F_{\nu} \in L^{\infty}(0,T)$. 
Since $\nu-4\nu^{2} > 0 $ for all $\nu \in (0,1/4)$, it follows from \cref{prop:ineq} that
\begin{align}  
  \esssup_{\nu \in (0,1/4)} \|F_{\nu}\|_{L^{\infty}(0,T)} < \infty. \label{eq:Fnu-uniform-bound}
\end{align}
In particular, recalling  $k * \ell \equiv 1$ on $(0,T)$ and  using  Young's convolution  inequality, we obtain 
\begin{align}
  &\esssup_{\nu \in (0,1/4)}
    \|\mathcal{B}(u_{\nu}-u_{0})\|_{L^{2}(0,T;H)}^{2}
  =
  \esssup_{\nu \in (0,1/4)}
    \sup_{t \in (0,T)} \int_{0}^{t} \|\mathcal{B}(u_{\nu}-u_{0})(s)\|_{H}^{2} \, \textup{d}s 
  \notag \\ 
  &\quad =
  \esssup_{\nu \in (0,1/4)}
    \|\,(k * \ell) * \|\mathcal{B}(u_{\nu}-u_{0})\|_{H}^{2}\,\|_{L^{\infty}(0,T)}
  \notag \\
  &\quad \leq \|k\|_{L^{1}(0,T)}
    \,\esssup_{\nu \in (0,1/4)}
    \|\ell * \|\mathcal{B}(u_{\nu}-u_{0})\|_{H}^{2}\|_{L^{\infty}(0,T)} < \infty.  \label{eq:B-uniform-L2-bound}
\end{align} 
Moreover, from \eqref{eq:energy-ineq-Fnu}, \eqref{eq:Fnu-uniform-bound}  
and the fact that 
$\nu - 4\nu^{2} \geq (1/2) \nu$ for all $\nu \in (0,1/8)$, we also deduce that 
\begin{align}
  &\esssup_{\nu \in (0,1/8)} 
    \nu^{-1} \|\nu \mathcal{A}(u_{\nu}-u_{0})\|_{L^{2}(0,T;H)}^{2}  
  \notag \\
  &\quad =
  \esssup_{\nu \in (0,1/8)} 
    \sup_{t \in (0,T)} \int_{0}^{t} \nu \| [\mathcal{A}(u_{\nu}-u_{0})](s) \|_{H}^{2} \, \textup{d}s
  < \infty. 
  \label{eq:A-uni-estimate}
\end{align}
Combining  \eqref{eq:star-nu},  
\eqref{eq:B-uniform-L2-bound} and \eqref{eq:A-uni-estimate}, 
 we can deduce that    
\begin{align}
  \esssup_{\nu \in (0,1/8)} 
   \|\partial \Phi_{\nu}(u_{\nu})\|_{L^{2}(0,T;H)}^{2}  
  &=
  \esssup_{\nu \in (0,1/8)} 
  \|\partial \varphi^{ \bullet }_{\nu}(u_{\nu}( \bullet ))\|_{L^{2}(0,T;H)}^{2}  
  < \infty.
  \label{eq:subdiff-uniform-bound}
\end{align}

\subsection{Convergence of approximate solutions} \label{subsec:conv-approx}
We first prove that
\begin{align*}
  \lim_{\mu,\nu\to 0+}\|u_{\mu}-u_{\nu}\|_{L^{2}(0,T;H)}=0,
\end{align*}
that is, the family $(u_{\nu})_{\nu\in(0,1/8)}$ forms a Cauchy sequence in $L^{2}(0,T;H)$.  
Fix arbitrary $\mu,\nu\in(0,1/8)$, and let $u_{\mu}$ and $u_{\nu}$ denote the solutions to  \eqref{eq:star-nu}  corresponding to the parameters $\mu$ and $\nu$, respectively.  
From  \eqref{eq:star-nu}  and $u_{\mu}-u_{\nu}\in D(\mathcal{A})\subset D(\mathcal{B})$, we get 
\begin{align*}
  \mathcal{B}(u_{\mu}-u_{\nu})(t) 
  &=
  -\,\mu\,\mathcal{A}(u_{\mu}-u_{0})(t)
  +\,\nu\,\mathcal{A}(u_{\nu}-u_{0})(t)
  \\ &\quad
  -\bigl(\partial \varphi^{t}_{\mu}(u_{\mu}(t))-\partial \varphi^{t}_{\nu}(u_{\nu}(t))\bigr)
\end{align*}
for a.e.\ $t \in (0,T)$.  
Multiplying both sides by $u_{\mu}-u_{\nu}\in L^{2}(0,T;H)$ and integrating  it  over $(0,t)$, we deduce from \cref{P:frac_chain} and H\"older's inequality that 
\begin{align}
  & \dfrac{1}{2}\bigl[k*\|u_{\mu}-u_{\nu}\|_{H}^{2}\bigr](t)  
  \notag
  \\
  &\quad \leq % [1 line]
  \int_{0}^{t} \bigl( \mathcal{B}(u_{\mu}-u_{\nu})(\tau) ,  u_{\mu}(\tau) - u_{\nu}(\tau) \bigr)_{H}  \, \textup{d}\tau 
  \notag
  \\ 
  &\quad =
  \int_{0}^{t} \bigl( -\mu \mathcal{A}(u_{\mu}-u_{0})(\tau) + \nu \mathcal{A}(u_{\nu}-u_{0})(\tau),\, u_{\mu}(\tau) -u_{\nu}(\tau) \bigr)_{H} \, \textup{d}\tau
  \notag
  \\ 
  &\qquad 
  - \int_{0}^{t}\bigl(\partial\varphi^{\tau}_{\mu}(u_{\mu}(\tau))-\partial\varphi^{\tau}_{\nu}(u_{\nu}(\tau)),\,u_{\mu}(\tau)-u_{\nu}(\tau)\bigr)_{H}\,\textup{d}\tau
  \notag
  \\ 
  &\quad \leq % [2–4 line]
  \|\mu\mathcal{A}(u_{\mu}-u_{0})\|_{L^{2}(0,T;H)} \, \|u_{\mu}-u_{\nu}\|_{L^{2}(0,T;H)} 
  \notag
  \\ 
  &\qquad 
  + \|\nu \mathcal{A}(u_{\nu}-u_{0})\|_{L^{2}(0,T;H)} \, \|u_{\mu}-u_{\nu}\|_{L^{2}(0,T;H)}
  \notag
  \\ 
  &\qquad 
  + \dfrac{\mu+\nu}{4}\Bigl(\|\partial \Phi_{\mu}(u_{\mu}) \|_{L^{2}(0,T;H)}^{2}  
  + \| \partial \Phi_{\nu}(u_{\nu}) \|_{L^{2}(0,T;H)}^{2}\Bigr)
  \label{eq:mu-nu-energy-ineq}
\end{align}
for~a.e.~$t \in (0,T)$, 
where we used K\={o}mura's  trick  (see, e.g., \cite[p.~56]{B-Brezis-1973}, \cite[p.~174]{B-Showalter-1997}), 
\begin{align*}
  &\bigl(\partial\varphi^{t}_{\mu}(a)-\partial\varphi^{t}_{\nu}(b),\, a-b \bigr)_{H}
  \\ &\quad \geq 
  -\dfrac{\mu+\nu}{4}\Bigl(\|\partial\varphi^{t}_{\mu}(a)\|_{H}^{2}
  + \|\partial\varphi^{t}_{\nu}(b)\|_{H}^{2}\Bigr)
\end{align*}
for all $a,b \in H$ and all $t\in(0,T)$. 
Since $k * \ell \equiv 1$ on $(0,T)$ and $(1 * \ell)(t) \leq \|\ell\|_{L^{1}(0,T)}$ for a.e.\ $t \in (0,T)$, convolving both sides of \eqref{eq:mu-nu-energy-ineq} with $\ell$ and applying Young's inequality, 
we obtain  
\begin{align*} 
&\dfrac{1}{2}\int_{0}^{t} \| u_{\mu}(\tau) - u_{\nu}(\tau) \|_{H}^{2} \, \textup{d}\tau
= 
\dfrac{1}{2} [  (\ell * k)  *  \| u_{\mu} - u_{\nu} \|_{H}^{2}] (t) 
\\ 
& \quad \leq
\|\mu \mathcal{A}(u_{\mu}-u_{0})\|_{L^{2}(0,T;H)} 
\, \|u_{\mu}-u_{\nu}\|_{L^{2}(0,T;H)} \, \| \ell \|_{L^{1}(0,T)}
\\ & \qquad + 
\|\nu \mathcal{A}(u_{\nu}-u_{0})\|_{L^{2}(0,T;H)} 
\, \|u_{\mu}-u_{\nu}\|_{L^{2}(0,T;H)} \, \| \ell \|_{L^{1}(0,T)}
\\ & \qquad + 
\left( 
\dfrac{\mu+\nu}{4} \|\partial \Phi_{\mu}(u_{\mu}) \|_{L^{2}(0,T;H)}^{2}  
+ \dfrac{\mu+\nu}{4} \|\partial \Phi_{\nu}(u_{\nu}) \|_{L^{2}(0,T;H)}^{2}  
\right)
\| \ell \|_{L^{1}(0,T)}
\\ 
& \quad \leq
2 \bigl( 
\|\mu \mathcal{A}(u_{\mu}-u_{0})\|_{L^{2}(0,T;H)} 
\, \|\ell\|_{L^{1}(0,T)} 
\bigr)^{2} 
+ 
\dfrac{1}{8} \|u_{\mu}-u_{\nu}\|_{L^{2}(0,T;H)}^{2} 
\\ & \qquad + 
2 \bigl( 
\|\nu \mathcal{A}(u_{\nu}-u_{0})\|_{L^{2}(0,T;H)} 
\, \|\ell\|_{L^{1}(0,T)} 
\bigr)^{2}
+ 
\dfrac{1}{8} \|u_{\mu}-u_{\nu}\|_{L^{2}(0,T;H)}^{2} 
\\ & \qquad + 
\left( 
\dfrac{\mu+\nu}{4} \|\partial \Phi_{\mu}(u_{\mu}) \|_{L^{2}(0,T;H)}^{2}  
+ \dfrac{\mu+\nu}{4} \|\partial \Phi_{\nu}(u_{\nu}) \|_{L^{2}(0,T;H)}^{2}  
\right)
\|\ell\|_{L^{1}(0,T)} 
\end{align*}
for a.e.\ $t \in (0,T)$. 
Taking the supremum  of both sides  over $t \in (0,T)$, one can take a constant $c_{0} \in [0,\infty)$ independent of $\mu$ and $\nu$ such that 
\begin{align*}
  &\|u_{\mu}-u_{\nu}\|_{L^{2}(0,T;H)}^{2} 
  \\ &\quad \leq 
  c_{0} \Bigl(\|\mu\mathcal{A}(u_{\mu}-u_{0})\|_{L^{2}(0,T;H)}^{2}
  + \|\nu\mathcal{A}(u_{\nu}-u_{0})\|_{L^{2}(0,T;H)}^{2}\Bigr)
  \\
  &\qquad + 
  c_{0} (\mu+\nu)\Bigl(\|\partial\Phi_{\mu}(u_{\mu})\|_{L^{2}(0,T;H)}^{2}
  + \|\partial\Phi_{\nu}(u_{\nu})\|_{L^{2}(0,T;H)}^{2}\Bigr).
\end{align*}
From \eqref{eq:A-uni-estimate} and \eqref{eq:subdiff-uniform-bound}, we obtain $\lim_{\mu,\nu\to 0_{+}}\|u_{\mu}-u_{\nu}\|_{L^{2}(0,T;H)}=0$.   
Hence there exists a function $u\in L^{2}(0,T;H)$ such that 
\begin{align}
  u_{\nu}\to u \quad \textup{in $L^{2}(0,T;H)$ as $\nu\to 0_{+}$}. 
  \label{eq:u-nu-conv-L2}
\end{align}

Furthermore, combining {\eqref{eq:subdiff-uniform-bound} } with  \eqref{eq:u-nu-conv-L2}  and using the identity 
$u_{\nu}( \bullet )  - J_{\nu}^{ \bullet }(u_{\nu}( \bullet ))=\nu\,\partial\varphi_{\nu}^{ \bullet }(u_{\nu}( \bullet ))$ 
(which follows from the definition of the Yosida approximation), 
we can deduce that  
\begin{align} 
  \textup{$J_{\nu}^{ \bullet }(u_{\nu}( \bullet )) \to u( \bullet )$ \quad in $L^{2}(0,T;H)$ as $\nu\to 0_{+}$.}
  \label{eq:Yosida-conv}
\end{align}  
By virtue of \eqref{eq:B-uniform-L2-bound}, \eqref{eq:subdiff-uniform-bound} and the reflexivity of  $L^{2}(0,T;H)$,  
there exist functions $\zeta_{1}, \zeta_{2} \in L^{2}(0,T;H)$ and a subsequence (not relabeled) such that  
\begin{align}
  \mathcal{B}(u_{\nu}-u_{0}) &\rightharpoonup \zeta_{1} 
  \quad \text{weakly in } L^{2}(0,T;H) 
  \quad \text{as } \nu \to 0_{+}, 
  \label{eq:weak-conv-B} \\
  \partial\Phi_{\nu}(u_{\nu}) &\rightharpoonup \zeta_{2} 
  \quad \text{weakly in } L^{2}(0,T;H) 
  \quad \text{as } \nu \to 0_{+}. 
  \label{eq:weak-conv-B-Phi}
\end{align} 
Passing to the limit in \eqref{eq:star-nu} and using \eqref{eq:A-uni-estimate}, \eqref{eq:weak-conv-B} and \eqref{eq:weak-conv-B-Phi}, we obtain  
\begin{align*}
  \zeta_{1}(t) + \zeta_{2}(t) = f(t) 
  \quad \text{in } H 
  \quad \text{for a.e.~} t \in (0,T).
\end{align*} 
Since $\mathcal{B}$ is maximal monotone and demiclosed (see \cref{subsec:subdifferential}), 
it follows from \eqref{eq:u-nu-conv-L2} and  \eqref{eq:weak-conv-B}  that  
$u - u_{0} \in D(\mathcal{B})$ and $\zeta_{1} = \mathcal{B}(u - u_{0})$.  
Furthermore,  by virtue of  \cref{prop:JY,prop:Phi-t-properties}, we have 
$\partial\varphi^{ \bullet }_{\nu}(u_{\nu}( \bullet )) \in \partial\Phi\bigl(J_{\nu}^{ \bullet }(u_{\nu}( \bullet ))\bigr)$.  
Since $\partial\Phi$ is also maximal monotone and demiclosed (see \cref{subsec:subdifferential}),  
we can deduce from \eqref{eq:Yosida-conv} and \eqref{eq:weak-conv-B-Phi} that  
$\zeta_{2} \in \partial\Phi(u)$.  
From \cref{prop:Phi-t-properties}, we obtain $\zeta_{2}(t) \in \partial\varphi^{t}(u(t))$ for a.e.\ $t \in (0,T)$.  

Therefore, $u$ is a strong solution to \eqref{E:main-equation1}.  
Moreover, due to \cref{thm:wellposedness}, the strong solution is unique.  
 Since $\bigl|\varphi^{t} \bigl( J_{\nu}^{t}(u_{\nu}(t)) \bigr)\bigr| \leq |\varphi_{\nu}^{t}(u_{\nu}(t))|$ for~a.e.~$t \in (0,T)$ (see \cref{prop:MY}), it follows from \eqref{eq:Fnu-uniform-bound} and \eqref{eq:Yosida-conv} that \eqref{eq:energy-estimate-strong-Sobolev} holds.         
Furthermore, by virtue of \eqref{eq:B-repr} and $\ell* \|\partial_{t}[k*(u-u_{0})]\|_{H}^{2} \in L^{\infty}(0,T)$,  $u$ belongs to $C([0,T];H)$ and $u(0) = u_{0}$ (see \cite[Lemma~4.3]{P-AkagiNakajima-2025}).  
This completes the proof.  \qed

\section{Proof of \cref{thm:main2} }  \label{Sec: proof of existence of strong solutions for Lebesgue spaces}  
In this section, we provide a  proof of \cref{thm:main2}.    
We define $\mathcal{B}$ as in \cref{subsec:time-nonlocal-ops}, 
and let $\Phi$ denote the proper lower-semicontinuous convex functional on $L^{2}(0,T;H)$ defined as in \cref{prop:Phi-t-properties}. %Let $u_{0} \in D(\varphi^{0})$ and $f \in L^{2}(0,T;H)$ be fixed. 

Let $(f_{n})$ be a sequence in $W^{1,2}(0,T;H)$ such that $f_{n} \to f$ strongly in $L^{2}(0,T;H)$.  
Then,  from  \cref{thm:main1}, there exist functions $u_{n}, \xi_{n} \in L^{2}(0,T;H)$ such that 
$\varphi^{\bullet}(u_{n}( \bullet )) \in L^{\infty}(0,T)$ and $(u_{n},\xi_{n}) \in \eqref{E:main-equation1}_{u_{0},f_{n}}$,   
that is, $k*(u_{n}-u_{0}) \in W^{1,2}(0,T;H)$ with $[k*(u_{n}-u_{0})](0) = 0$ (i.e., $u_{n}-u_{0} \in D(\mathcal{B})$),  
$\xi_{n}(t) \in \partial \varphi^{t}(u_{n}(t))$ for~a.e.~$t \in (0,T)$ (equivalently, $\xi_{n} \in \partial \Phi(u_{n})$), and moreover, $u_{n}$ and $\xi_{n}$ satisfy the equation, 
\begin{equation}
  \mathcal{B}(u_{n}-u_{0})(s) + \xi_{n}(s) = f_{n}(s)
  \quad \text{for a.e.\ } s \in (0,T).
  \label{eq:fn-approx-eq}
\end{equation}
 We first prove \cref{thm:main2} under the additional assumption,    
\begin{equation}\label{eq:nonneg-varphi}
  \varphi^{t}(z) \geq 0
  \quad \textup{for all $t \in [0,T]$ and all $z \in H$.}
\end{equation} 

Multiplying both sides of \eqref{eq:fn-approx-eq} by $\xi_{n} \in L^{2}(0,T;H)$, integrating it over $(0,t)$, and employing the fractional chain-rule formula obtained in \cref{lem:chainrule2}, for every $\varepsilon \in (0,1)$, 
we  deduce  from \eqref{eq:nonneg-varphi} and Young's inequality that 
\begin{align*} 
  %[1–2 line]
  &
  \bigl[k*\varphi^{\bullet}(u_{n}( \bullet ))\bigr](t) - \varphi^{0}(u_{0}) \int_{0}^{t} k(s)\,\textup{d}s  
  - \varepsilon  C  
    \int_{0}^{t}\!\|\xi_{n}(s)\|_{H}^{2}\,\textup{d}s  \notag \\
  &\qquad
  -  \dfrac{ C }{\varepsilon}  \left[ T \bigl( 1 + \varphi^{0}(u_{0}) \bigr) + \int_{0}^{t} \varphi^{s}(u_{n}(s))\,\textup{d}s \right]   
  + \int_{0}^{t}\!\|\xi_{n}(s)\|_{H}^{2}\,\textup{d}s  \notag \\
  %[3 line]  
  &\quad  \leq 
  \int_{0}^{t} \Bigl( \mathcal{B}(u_{n}-u_{0})(s) + \xi_{n}(s),\, \xi_{n}(s) \Bigr)_{H}\,\textup{d}s    \notag \\
  %[4 line]  
  &\quad  = 
  \int_{0}^{t} \bigl( f_{n}(s),\, \xi_{n}(s) \bigr)_{H}\,\textup{d}s 
  \leq 
  \dfrac{1}{2\varepsilon} \int_{0}^{t} \| f_{n}(s)\|_{H}^{2}\,\textup{d}s 
  + 
  \dfrac{\varepsilon}{2} \int_{0}^{t} \|\xi_{n}(s)\|_{H}^{2}\,\textup{d}s  \notag \\
  %[5 line]  
  &\quad  \leq 
  \dfrac{1}{2\varepsilon} \sup_{n \in \mathbb{N}} \| f_{n}\|_{L^{2}(0,T;H)}^{2} 
  + 
  \dfrac{\varepsilon}{2} \int_{0}^{t} \|\xi_{n}(s)\|_{H}^{2}\,\textup{d}s
\end{align*} 
for~a.e.~$t \in (0,T)$.  
Here, $C \in [0,\infty)$ denotes the constant satisfying \eqref{eq:device-4} and independent of $\varepsilon$  and $n$.  
Hence, choosing $\varepsilon$ sufficiently small and using $k*\ell \equiv 1$ on $(0,T)$, 
 we can take  a constant $ C_{0} \in [0,\infty)$  independent of $n$ such that 
\begin{align} 
&\bigl[k*\varphi^{ \bullet }(u_{n}( \bullet ))\bigr](t) + \int_{0}^{t} \| \xi_{n}(s)\|_{H}^{2}\,\textup{d}s  \notag 
\\ 
& \quad  \leq 
C_{0} \left( 1 + \int_{0}^{t} \varphi^{s}(u_{n}(s))\,\textup{d}s \right) 
=
C_{0} + C_{0} \left[ \ell * \bigl( k*\varphi^{ \bullet } (u_{n}(\bullet)) \bigr)\right](t)
\label{eq:uniform-estimate-un}
\end{align} 
for~a.e.~$t \in (0,T)$.  
Since $\varphi^{\bullet}(u_{n}(\bullet)) \in L^{\infty}(0,T)$ and $k \in L^{1}(0,T)$,  
we have $k * \varphi^{\bullet}(u_{n}(\bullet)) \in L^{\infty}(0,T)$.  
Therefore, applying \cref{prop:ineq} (see also \cite[Lemma~4.5]{P-AkagiNakajima-2025}), we obtain  
\begin{align}
  \sup_{n \in \mathbb{N}} \| k * \varphi^{\bullet}(u_{n}(\bullet)) \|_{L^{\infty}(0,T)} < \infty.
  \label{eq:kphi-uniform-Linf}
\end{align} 
Moreover, combining \eqref{eq:uniform-estimate-un} with \eqref{eq:kphi-uniform-Linf}, 
we get   
\begin{align} 
\sup_{n \in \mathbb{N}} \|\xi_{n}\|_{L^{2}(0,T;H)}^{2}
=
\sup_{n \in \mathbb{N}} \left\| \int_{0}^{ \bullet } \|\xi_{n}(s)\|_{H}^{2}\,\textup{d}s \right\|_{L^{\infty}(0,T)} 
< \infty.  
\label{eq:xi-n-uniform-L2:n}
\end{align} 
Furthermore, since \eqref{eq:fn-approx-eq} holds and  
$\sup_{n \in \mathbb{N}} \| f_{n}\|_{L^{2}(0,T;H)} < \infty$,  
we also have  
\begin{align} 
\sup_{n \in \mathbb{N}} \|\mathcal{B}(u_{n}-u_{0})\|_{L^{2}(0,T;H)} < \infty.  
\label{eq:B-uniform-L2-bound:n}
\end{align}
Next, we prove that $(u_{n})$ forms a Cauchy sequence in $L^{2}(0,T;H)$.  
Fix arbitrary $n,m \in \mathbb{N}$.  
Since $(u_{n}, \xi_{n}) \in \eqref{E:main-equation1}_{u_{0},f_{n}}$ and $(u_{m}, \xi_{m}) \in \eqref{E:main-equation1}_{u_{0},f_{m}}$,  
it follows from \cref{thm:wellposedness} that there exists a constant $C_{1} \in [0,\infty)$  
independent of $n$ and $m$ such that  
\begin{align*} 
\| u_{n} - u_{m}\|_{L^{2}(0,T;H)}^{2}  
\leq C_{1} \| f_{n} - f_{m}\|_{L^{2}(0,T;H)}^{2}. 
\end{align*}
Since $f_{n} \to f$ strongly in $L^{2}(0,T;H)$,  
we conclude that $\lim_{n,m \to \infty} \| u_{n} - u_{m}\|_{L^{2}(0,T;H)} = 0$.  
Hence there exists a function $u\in L^{2}(0,T;H)$ such that 
\begin{align*}
  u_{n} \to u \quad \textup{in $L^{2}(0,T;H)$ as $n \to \infty$}. 
  %\label{eq:u-n-conv-L2}
\end{align*} 
From \eqref{eq:xi-n-uniform-L2:n}, \eqref{eq:B-uniform-L2-bound:n} and the reflexivity of  $L^{2}(0,T;H)$,   
there exist functions $\zeta_{1}, \zeta_{2} \in L^{2}(0,T;H)$ and a subsequence (not relabeled) such that  
\begin{align*}
  \mathcal{B}(u_{n}-u_{0}) &\rightharpoonup \zeta_{1} 
  \quad \text{weakly in } L^{2}(0,T;H) 
  \quad \text{as } n \to \infty,  
  \notag \\
  \xi_{n} &\rightharpoonup \zeta_{2} 
  \quad \text{weakly in } L^{2}(0,T;H) 
  \quad \text{as } n \to \infty.  
\end{align*} 
Passing to the limit in \eqref{eq:fn-approx-eq}, we obtain  
\begin{align*}
  \zeta_{1}(t) + \zeta_{2}(t) = f(t) 
  \quad \text{in } H 
  \quad \text{for a.e.~} t \in (0,T).
\end{align*}
Repeating the same argument as in \cref{subsec:conv-approx}, we conclude that $u$ is a strong solution to \eqref{E:main-equation1}.  
Hence the proof is complete under the assumption \eqref{eq:nonneg-varphi}.

Finally, we treat the general case.  
From \eqref{eq:td-subdiff-varphi-bound}, 
%and \cref{prop:time-subdiff}, 
there exists a constant $D_{0} \in [0,\infty)$ such that  
\begin{align} 
  \varphi^{t}(z) + \frac{D_{0}}{2} \| z\|_{H}^{2} + D_{0} 
  \geq |\varphi^{t}(z)| + 1 \geq 0
\end{align}
for all $t \in [0,T]$ and all $z \in H$. 
For each $t \in [0,T]$, we define a functional $\psi^{t} \colon H \to (-\infty,\infty]$ by  
\begin{align*} 
  \psi^{t}(z) := \varphi^{t}(z) + \frac{D_{0}}{2}\|z\|_{H}^{2} + D_{0} 
  \quad \textup{for $z \in H$.} 
\end{align*}
Then $\psi^{t}$ is the proper lower-semicontinuous convex functional, and the condition
\eqref{eq:nonneg-varphi} holds when $\varphi^{t}$ is replaced by $\psi^{t}$. 
Furthermore, the family $\{\psi^{t}\}_{t \in [0,T]}$ satisfies assumptions \textup{(A1)} and \textup{(A2)}.  
Indeed, let $s,t \in [0,T]$, $a,b \in H$ and $A \in [0,\infty)$ satisfy  
\begin{align}
  \|b-a\|_{H}
  &\leq
  A|t-s|\bigl(1+|\varphi^{s}(a)|\bigr)^{1/2},
  \label{eq:A-1-psi-condition} \\
  \varphi^{t}(b)
  &\leq
  \varphi^{s}(a) + A|t-s|\bigl(1+|\varphi^{s}(a)|\bigr). 
  \label{eq:A-2-psi-condition}
\end{align}
Then Young's inequality yields  
\begin{align*}
  \psi^{t}(b)
  &= \varphi^{t}(b)
     + \frac{D_{0}}{2}\,\| b\|_{H}^{2}
     + D_{0}
\\
  &\leq 
     \varphi^{s}(a)
     + A|t-s| \bigl(1+|\varphi^{s}(a)|\bigr)
     + \frac{D_{0}}{2}\,\|a\|_{H}^{2}
     + \frac{D_{0}}{2}\!\left( \| b\|_{H}^{2}-\| a\|_{H}^{2} \right)
     + D_{0}
\\
  &= 
     \psi^{s}(a)
     + A|t-s| \bigl(1+|\varphi^{s}(a)|\bigr)
     + \frac{D_{0}}{2}\!\left( \| b\|_{H} + \| a\|_{H} \right)
      \left( \| b\|_{H} - \| a\|_{H} \right)
\\
  &\leq      
     \psi^{s}(a)
     + A|t-s| \bigl(1+|\psi^{s}(a)|\bigr)
     + \frac{D_{0}}{2}\!\left( \| b-a\|_{H} + 2 \| a\|_{H} \right)
       \| b - a\|_{H} 
\\
  &\leq          
     \psi^{s}(a)
     + A|t-s| \bigl(1+|\psi^{s}(a)|\bigr)
     + \frac{D_{0}}{2} \Bigl( A |t-s|
       \bigl( 1+ |\varphi^{s}(a)|\bigr)^{1/2} \Bigr)^{2}
\\ &\qquad
     + D_{0} \, \| a\|_{H} \, A |t-s|  
       \bigl( 1+ |\varphi^{s}(a)| \bigr)^{1/2}
\\
  &=          
     \psi^{s}(a)
     + A|t-s| \bigl(1+|\psi^{s}(a)|\bigr)
     + \frac{D_{0}}{2} A^{2} |t-s|^{2}
       \bigl( 1+ |\varphi^{s}(a)|\bigr)
\\ &\qquad
     + \sqrt{2 D_{0}}\,A\,|t-s|\!
       \left( \frac{D_{0}}{2}\|a\|_{H}^{2} \right)^{1/2}
       \bigl( 1+ |\varphi^{s}(a)| \bigr)^{1/2}
\\
  &\leq               
     \psi^{s}(a)
     + A|t-s| \bigl(1+|\psi^{s}(a)|\bigr)
     + D_{0} A^{2} T |t-s|
       \bigl( 1+ |\psi^{s}(a)|\bigr)
\\ &\qquad
     + \sqrt{2 D_{0}}\,A\,|t-s|  
      \bigl| \psi^{s}(a) \bigr|^{1/2}
       \bigl( 1+ |\psi^{s}(a)| \bigr)^{1/2} 
\\
  &\leq      
     \psi^{s}(a)
     + \left( A + D_{0} A^{2} T + \sqrt{2 D_{0}}\,A \right) |t-s| \bigl(1 + |\psi^{s}(a)|\bigr)
\end{align*}
for all $s,t \in [0,T]$, $a,b \in H$ and $A \in [0,\infty)$ satisfying \eqref{eq:A-1-psi-condition} and \eqref{eq:A-2-psi-condition}. 
This shows that the family $\{\psi^{t}\}_{t \in [0,T]}$ satisfies assumptions \textup{(A1)} and \textup{(A2)}.
Thus, for every $g \in L^{2}(0,T;H)$, there exist unique functions $v,\zeta \in L^{2}(0,T;H)$ such that  
$v-u_{0} \in D(\mathcal{B})$ (see \cref{subsec:time-nonlocal-ops}) and 
\begin{align} 
  \zeta(t) \in \partial \psi^{t}\bigl( v(t) \bigr), 
  \qquad 
  \mathcal{B}(v-u_{0})(t) + \zeta(t) = g(t)
  \label{eq:psi-u0-g} \tag{PP}
\end{align}
for a.e.\ $t \in (0,T)$.  
We write $(v,\zeta) \in \eqref{eq:psi-u0-g}_{u_{0},\,g}$ if 
$v,\zeta \in L^{2}(0,T;H)$ satisfy the above conditions.
We next define $\Lambda \colon L^{2}(0,T;H) \to L^{2}(0,T;H)$ as follows.
For each $h \in L^{2}(0,T;H)$, denote by $(v,\zeta)$ the unique pair in 
$L^{2}(0,T;H)\times L^{2}(0,T;H)$ such that 
$(v,\zeta) \in \eqref{eq:psi-u0-g}_{u_{0},\,f+D_{0}h}$. 
We then set $\Lambda(h) := v$. 
Here, a straightforward computation shows that the mapping $w \mapsto \frac{D_{0}}{2}\,\|w\|_{H}^{2} + D_{0}$
is convex and Fr\'echet differentiable in $H$, and its Fr\'echet derivative is given by  
$D_{0}\, I_{H}$, where $I_{H} \colon H \to H$ denotes the identity mapping on $H$. 
Thus, for each $t \in [0,T]$, we have $D(\partial\varphi^{t}) = D(\partial\psi^{t})$  
and 
\begin{equation*}
  \partial\psi^{t}(w)
  = \partial\varphi^{t}(w) + \{ D_{0}\, w \}
\end{equation*}
for all $w \in D(\partial\varphi^{t}) = D(\partial\psi^{t})$ (see, e.g., \cite[Corollary~2.11]{B-Brezis-1973}, \cite[Chapter~II, Propositions 7.6 and 7.7]{B-Showalter-1997}). %\cite[Propositions~4.4.29 and 4.4.31]{B-PapageorgiouGasinski-2006}
In order to complete the proof, it suffices to show that there exists a unique function 
$v_{*} \in L^{2}(0,T;H)$ such that $\Lambda(v_{*}) = v_{*}$. 
Let $\beta \in (0,\infty)$ be a constant which will be determined later, and set 
$\mathfrak{X}_{\beta} := L^{2}(0,T;H)$ equipped with a norm given by  
\begin{align*}
  \| w \|_{\mathfrak{X}_{\beta}}
  :=
  \esssup_{t \in (0,T)}
  \left|
     \mathrm{e}^{- \beta t}
     \int_{0}^{t} \| w(s)\|_{H}^{2}\, \mathrm{d}s
  \right|
  =
  \bigl\|
     \mathrm{e}^{-\beta \,\bullet}
     \bigl[1*\|w\|_{H}^{2}\bigr](\bullet)
  \bigr\|_{L^{\infty}(0,T)}
\end{align*}
for $w \in \mathfrak{X}_{\beta}$. 
Then $(\mathfrak{X}_{\beta}, \|\bullet\|_{\mathfrak{X}_{\beta}})$ is a Banach space. 
Let $\kappa_{0} \in (0,1)$ be fixed.  
Then we can choose $\beta \in (0,\infty)$ sufficiently large such that
\begin{align} \label{eq:Lambda-contraction}
  \|\Lambda(w_{1}) - \Lambda(w_{2})\|_{\mathfrak{X}_{\beta}}
  \leq
  \kappa_{0}\, \|w_{1} - w_{2}\|_{\mathfrak{X}_{\beta}}
\end{align}
for all $w_{1}, w_{2} \in \mathfrak{X}_{\beta} = L^{2}(0,T;H)$.
Indeed, let $h_{1}, h_{2} \in L^{2}(0,T;H)$, and let 
$(v_{i},\zeta_{i}) \in \eqref{eq:psi-u0-g}_{u_{0},\,f + D_{0} h_{i}}$ for $i=1,2$.
Then we have $v_{1}-v_{2} \in D(\mathcal{B})$ and 
\begin{gather}
  \mathcal{B}(v_{1}-v_{2})(t)
  + \zeta_{1}(t) - \zeta_{2}(t)
  =
  D_{0}\bigl(h_{1}(t) - h_{2}(t)\bigr),
  \label{eq:B-diff} \\[0.5em]
  \zeta_{1}(t) - \zeta_{2}(t)
  \in
  \partial\psi^{t}\bigl(v_{1}(t)\bigr)
  -
  \partial\psi^{t}\bigl(v_{2}(t)\bigr) 
  \label{eq:psi-t-mono}
\end{gather}
for a.e.\ $t \in (0,T)$. 
From \eqref{eq:psi-t-mono} and the monotonicity of $\partial \psi^{t}$ 
(see \cref{subsec:subdifferential}), we obtain 
\begin{align}
  \int_{0}^{t}
    \bigl(
      \zeta_{1}(\tau) - \zeta_{2}(\tau),\,
      v_{1}(\tau) - v_{2}(\tau)
    \bigr)_{H}
  \, \mathrm{d}\tau
  \;\geq\; 0
  \label{eq:mono-int}
\end{align}
for a.e. $t \in (0,T)$. 
Hence, multiplying both sides of \eqref{eq:B-diff} by $v_{1}-v_{2} \in L^{2}(0,T;H)$ and integrating it over $(0,t)$, we deduce from \eqref{eq:mono-int}, \cref{P:frac_chain} and Young's inequality that 
\begin{align*} 
  &\dfrac{1}{2} \bigl[ k * \| v_{1}-v_{2}\|_{H}^{2} \bigr](t) 
  \leq    
  \int_{0}^{t}  \bigl( \mathcal{B}(v_{1}-v_{2})(s), v_{1}-v_{2}(s) \bigr)_{H} \, \textup{d} s     
\\ & \quad \leq     
  \int_{0}^{t}  \bigl( \mathcal{B}(v_{1}-v_{2})(s) + \zeta_{1}(s) - \zeta_{2}(s), v_{1}-v_{2}(s) \bigr)_{H} \, \textup{d} s 
\\ & \quad =     
  D_{0} \int_{0}^{t}  \bigl( h_{1}(s)-h_{2}(s), v_{1}-v_{2}(s) \bigr)_{H} \, \textup{d} s
\\ & \quad \leq     
  \dfrac{D_{0}^{2}}{2} \int_{0}^{t}  \bigl\|h_{1}(s)-h_{2}(s)\bigr\|_{H}^{2} \, \textup{d} s 
  + 
  \dfrac{1}{2}\int_{0}^{t} \bigl\|v_{1}-v_{2}(s) \bigr\|_{H}^{2} \, \textup{d} s 
\\ & \quad =      
  \dfrac{D_{0}^{2}}{2} \left( 1*\bigl\|h_{1}-h_{2}\bigr\|_{H}^{2} \right)(t) 
  + 
  \dfrac{1}{2}  \left(1 * \bigl\|v_{1} - v_{2} \bigr\|_{H}^{2} \right)(t)   
\end{align*} 
for~a.e.~$t \in (0,T)$. 
Convolving both sides with $\ell$ and using $k*\ell \equiv 1$ on $(0,T)$, we get 
\begin{align*} 
  &\dfrac{1}{2} \bigl(1 * \| v_{1}-v_{2}\|_{H}^{2} \bigr)(t) 
  =
  \dfrac{1}{2} \bigl( (\ell * k) * \| v_{1}-v_{2}\|_{H}^{2} \bigr)(t) 
\\ & \quad \leq    
  \dfrac{D_{0}^{2}}{2} \left[ \ell * \left( 1*\bigl\|h_{1}-h_{2}\bigr\|_{H}^{2} \right) \right](t) 
  + 
  \dfrac{1}{2}  \left[ \ell * \left(1 * \bigl\|v_{1} - v_{2} \bigr\|_{H}^{2} \right) \right](t)
\end{align*} 
for~a.e.~$t \in (0,T)$. 
Therefore it follows from \eqref{eq:conv-exp-identity} and Young's convolution inequality that  
\begin{align*} 
  &\dfrac{1}{2} \bigl(1 * \| v_{1}-v_{2}\|_{H}^{2} \bigr)(t) \, \textup{e}^{- \beta t} 
\\ & \quad \leq    
  \dfrac{D_{0}^{2}}{2} \left[ \ell * \left( 1*\bigl\|h_{1}-h_{2}\bigr\|_{H}^{2} \right) \right](t) \, \textup{e}^{- \beta t}
  + 
  \dfrac{1}{2}  \left[ \ell * \left(1 * \bigl\|v_{1} - v_{2} \bigr\|_{H}^{2} \right) \right](t) \, \textup{e}^{- \beta t} 
\\ & \quad =
  \dfrac{D_{0}^{2}}{2} \left[ \left( \ell(\bullet) \, \textup{e}^{- \beta \bullet}  \right) * \biggl( \left( 1*\bigl\|h_{1}-h_{2}\bigr\|_{H}^{2} \right) (\bullet) \, \textup{e}^{- \beta \bullet} \biggr) \right](t) 
\\ & \qquad + 
  \dfrac{1}{2}  \left[ \left( \ell(\bullet) \, \textup{e}^{- \beta \bullet}  \right) * \biggl( \left( 1*\bigl\|v_{1}-v_{2}\bigr\|_{H}^{2} \right) (\bullet) \, \textup{e}^{- \beta \bullet} \biggr) \right](t) 
\\ & \quad \leq      
  \dfrac{D_{0}^{2}}{2} \left\| \ell(\bullet) \, \textup{e}^{- \beta \bullet} \right\|_{L^{1}(0,T)} \left\| h_{1}-h_{2} \right\|_{\mathfrak{X}_{\beta}} 
  +
  \dfrac{1}{2} \left\| \ell(\bullet) \, \textup{e}^{- \beta \bullet} \right\|_{L^{1}(0,T)} \left\| v_{1}-v_{2} \right\|_{\mathfrak{X}_{\beta}} 
\end{align*}
for~a.e.~$t \in (0,T)$. 
\begin{comment}
Here we used that, for every $g \in L^{1}(0,T)$ and $w \in L^{1}(0,T)$, the following identity holds:   
\begin{align*} 
  &(g*w)(t) \exp(- \beta t) 
  = 
  \int_{0}^{t} g(t-s) \exp\bigl(- \beta(t-s)\bigr) \exp(- \beta s) w(s) \, \textup{d}s 
  \notag \\& \quad =
  \bigl[ \bigl( g(\bullet) \exp(- \beta \, \bullet) \bigr) * \bigl(w(\bullet) \exp(-\beta \, \bullet)\bigr) \bigr](t)  %\label{eq:conv-exp-identity} 
\end{align*} 
for a.e.\ $t \in (0,T)$. 
\end{comment}
Taking the supremum of both sides over $(0,T)$, we obtain  
\begin{align*} 
  \dfrac{1}{2}\left( 1 - \left\| \ell(\bullet) \, \textup{e}^{- \beta \bullet} \right\|_{L^{1}(0,T)} \right) \left\| v_{1}-v_{2} \right\|_{\mathfrak{X}_{\beta}} 
  \leq 
  \dfrac{D_{0}^{2}}{2} \left\| \ell(\bullet) \, \textup{e}^{- \beta \bullet} \right\|_{L^{1}(0,T)} \left\| h_{1}-h_{2} \right\|_{\mathfrak{X}_{\beta}}.  
\end{align*}
Thus, we can choose a constant $\beta_{0} \in (0,\infty)$ sufficiently large such that \eqref{eq:Lambda-contraction} holds. 
In particular, $\Lambda$ is a contraction mapping on the Banach space $(\mathfrak{X}_{\beta_{0}}, \| \bullet \|_{\mathfrak{X}_{\beta_{0}}})$.  
By virtue of Banach's fixed point theorem, there exists a unique function $v_{*} \in \mathfrak{X}_{\beta_{0}} = L^{2}(0,T;H)$ satisfying $\Lambda(v_{*}) = v_{*}$.  
This completes the proof.  \qed 

%%%%%%%%%%%%%%%%%%%%%%%%%
%%%%%%%%%%%%%%%%%%%%%%%%%
%%%%%%%%%%%%%%%%%%%%%%%%%
%%%%%%%%%%%%%%%%%%%%%%%%%
%%%%%%%%%%%%%%%%%%%%%%%%%
%%%%%%%%%%%%%%%%%%%%%%%%%

\section{Application}  \label{Sec: Application}

In this section, we apply the abstract results obtained so far to the Cauchy--Dirichlet problem for certain nonlinear parabolic equations on moving domains.

We denote by $d \in \mathbb{N}$, by $x = (x_{1}, \ldots, x_{d})$ a generic point of $\mathbb{R}^{d}$, by $\gamma = (\gamma_{1}, \ldots, \gamma_{d})$ a multi-index and by $D_{x}^{\gamma}$ the differentiation,  
\begin{align*} 
  \dfrac{\partial^{|\gamma|}}{\partial x_{1}^{\gamma_{1}} \cdots \partial x_{d}^{\gamma_{d}}}  
\end{align*} 
where $|\gamma| = \sum_{j=1}^{d}\gamma_{j}$. 
Let $U \subset \mathbb{R}^{d}$ be a bounded domain with smooth boundary $\partial U$, and 
for each $t \in [0,T]$, let $\Omega_{t} \subset U$ be a bounded domain of $\mathbb{R}^{d}$ with smooth boundary $\partial\Omega_{t}$. 
We impose the following assumption: 
\begin{itemize}
\item[(B)] 
%There exists a bounded domain $U \subset \mathbb{R}^{d}$ with smooth boundary $\partial U$ such that 
It holds that $\overline{\Omega_{t}} \subset U$ for all $t \in [0,T]$. 
Furthermore, there exists a diffeomorphism $\Theta_{t} = (\theta_{1}^{t}, \ldots, \theta_{d}^{t})$ of class $C^{1}$ from $\overline{U}$ onto itself with $\Theta_{t}(\Omega_{0}) = \Omega_{t}$ for every $t \in [0,T]$ such that $\Theta_{0}$ is the identity on $\overline{U}$ and $D_{x}^{\gamma}\theta_{i}^{t}$ is continuously differentiable in $t$ on $[0,T] \times \overline{U}$ for every multi-index $\gamma$ with $|\gamma| \leq 1$ and $i = 1, \ldots, d$.  
\end{itemize} 
\noindent 
%We write $(\Omega_{t})_{t \in [0,T]} \in \textup{(B)}$ if the above assumption is satisfied. 
Let $Q  \subset [0,T] \times U$ and $\partial Q \subset [0,T] \times \overline{U}$ be defined by
\begin{align*}
  Q := \bigcup_{t \in [0,T]} \bigl(\{t\}\times\Omega_{t}\bigr), 
  \quad 
  \partial Q := \bigcup_{t \in [0,T]} \bigl(\{t\}\times\partial\Omega_{t}\bigr).  
\end{align*} 
We now consider the following Cauchy--Dirichlet problem:
\begin{equation}\tag{CDP}\label{PP}
\left. 
\begin{aligned}
  \partial_{t}^{\alpha} \bigl(u-u_{0}\bigr) - \Delta_{p} u &= f &&\textup{in } Q,\\
  u &= 0 &&\textup{on } \partial Q,  
\end{aligned}
\right\}
\end{equation}
where $\alpha \in (0,1)$, $p \in [2,\infty)$, and $u_{0} = u_{0}(x)$ and 
 $f=f(t,x)$ are prescribed.  
Moreover, $\partial_{t}^{\alpha}(u-u_{0}) := \partial_{t}\bigl[k_{1-\alpha}*(u-u_{0})\bigr]$ denotes the $\alpha$-th order Riemann--Liouville derivative of $u-u_{0}$, and $\Delta_{p}$ is the so-called \emph{$p$-Laplacian} given as $\Delta_{p}u := - \textup{div}( |\nabla u|^{p-2}\nabla u )$. 
For each $t \in [0,T]$, let $K(t) \subset W^{1,p}(U)$ be defined by 
\begin{align*} 
  K(t) := \{w \in W^{1,p}(U) \colon  w|_{\Omega_{t}} \in W_{0}^{1,p}(\Omega_{t}), \, w(x) = 0 \textup{ for~a.e.~} x \in U \setminus \Omega_{t}\}. 
\end{align*} 
We are concerned with the definition of $L^{2}$-solutions to \eqref{PP}.  
\begin{definition}[$L^{2}$-solutions to \eqref{PP}] \label{def:strong-sol} 
Let $u_{0} \in L^{2}(U)$, and let $f \in L^{2}(0,T;L^{2}(U))$.   
A function $u \in L^{2} \bigl(0,T;L^{2}(U)\bigr)$ is called an \emph{$L^{2}$-solution} to
\eqref{PP} if the following conditions are all satisfied\/\textup{:} 
\begin{itemize}
\item[\textup{(i)}] It holds that $k_{1-\alpha} * (u-u_0) \in W^{1,2} \bigl(0,T;L^{2}(U)\bigr)$, $[k_{1-\alpha} * (u-u_0)](0) = 0$ and $u(t, \bullet) \in K(t)$ for~a.e.~$t \in (0,T)$.  
\item[\textup{(ii)}] There exists a null set $N \subset (0,T)$ such that for all $t \in (0,T) \setminus N$, the following properties are satisfied\/\textup{:} $\tilde{u}(t,\bullet) := u(t,\bullet)|_{\Omega_{t}} \in W_{0}^{1,p}(\Omega_{t})$ and $\Delta_{p} \tilde{u}(t,\bullet) \in L^{2}(\Omega_{t})$, and 
\begin{equation*}
  \partial_t\!\bigl[k_{1-\alpha} * (u-u_0)\bigr](t, \bullet ) - \Delta_{p} \tilde{u}(t, \bullet )
  = f(t,\bullet) \quad \text{in } L^{2}(\Omega_{t}).
\end{equation*}
%for a.e.\ $t\in(0,T)$. 
\end{itemize}
\end{definition}
Let $\{\Omega_{t}\}_{t \in [0,T]}$ be a family of bounded domains such that \textup{(B)} holds, and set $H := L^{2}(U)$. 
For each $t \in [0,T]$, we define $\varphi_{p}^{t} \colon L^{2}(U) \to [0,\infty]$ by
\begin{equation*} \label{eq:energy-pt}
  \varphi_{p}^{t}(w) :=
  \begin{cases}
    \displaystyle \frac{1}{p} \int_{\Omega_{t}} |\nabla w|^{p}\,\textup{d}x
    &\text{if} \ w \in K(t), \\[1.25ex] 
    \infty &\text{otherwise} 
  \end{cases}
\end{equation*}
for $w \in L^{2}(U)$. 
Then for each $t \in [0,T]$, $\varphi_{p}^{t}$ is proper, lower-semicontinuous, and  convex in $L^{2}(U)$, and moreover, $\partial \varphi_{p}^{t}(w)$ coincides with $-\Delta_{p}w$ equipped with the homogeneous Dirichlet boundary condition in the distributional sense for $w \in D(\partial \varphi_{p}^{t})$, where 
\begin{align*} 
D(\partial \varphi_{p}^{t}) = \{w \in L^{2}(U) \cap K(t) \colon w|_{\Omega_{t}} \in W_{0}^{1,p}(\Omega_{t})\textup{ and } \Delta_{p} (w|_{\Omega_{t}}) \in L^{2}(\Omega_{t})\}.
\end{align*}
Hence the Cauchy--Dirichlet problem \eqref{PP} is reduced to  the following abstract Cauchy problem: 
\begin{equation*} 
  \partial_{t}\!\bigl[k_{1-\alpha} * (u-u_{0})\bigr](t) 
  + \partial \varphi_{p}^{t}\bigl(u(t)\bigr) = f(t)
  \quad \textup{in } L^{2}(U) \quad \textup{for $t \in (0,T)$.}
\end{equation*}

Moreover, the assumption (B) guarantees that both (A1) and (A2) are satisfied.  
Indeed, from the assumption (B), we can define a continuous mapping  
$\Psi_{p} \colon [0,T] \times [0,T] \times L^{2}(U) \to L^{2}(U)$ by  
$
\Psi_{p}(t,s,w(\bullet )) := w \bigl( \Theta(s,\Theta^{-1}(t,\bullet)) \bigr)
$ 
for $(t,s,w) \in [0,T] \times [0,T] \times L^{2}(U)$. 
Then, there exists a constant $C \in [0,\infty)$ such that  
\begin{align*} 
  \|\Psi_{p}(t,s,w) - w\|_{L^{2}(U)}  
  &\leq C\,|t-s| \bigl(1 + |\varphi_{p}^{s}(w)|\bigr)^{1/2}, \\
  \varphi_{p}^{t}\bigl(\Psi_{p}(t,s,w)\bigr)   
  &\leq \varphi_{p}^{s}(w) + C\,|t-s| \bigl(1 + |\varphi_{p}^{s}(w)|\bigr)
\end{align*}
for all $(t,s,w) \in [0,T] \times [0,T] \times L^{2}(U)$ (see \cite[Lemmas 3.2.2 and 3.2.3]{A-Kenmochi-1981}). 
Hence \textup{$(A\varphi^{t})_{NL}$} holds.  
From \cref{Remark: Sufficient condition for A1 and A2}, these estimates enable us to verify that conditions (A1) and (A2) are fulfilled.  
Thus we can apply \cref{thm:main1,thm:main2} to conclude the existence of $L^{2}$-solutions to \eqref{PP}.  

\begin{theorem}[Existence of $L^{2}$-solutions to \eqref{PP}] \label{thm:application}
Let $\{\Omega_{t}\}_{t \in [0,T]}$ be a family of bounded domains of $\mathbb{R}^{d}$ such that \textup{(B)} holds. Then, for every $u_{0}\in L^{2}(U)$ satisfying $u_{0}|_{\Omega_{0}} \in W^{1,p}_{0}(\Omega_{0})$ and $f \in L^{2}\bigl(0,T;L^{2}(U)\bigr)$, the Cauchy--Dirichlet problem \eqref{PP} admits a unique $L^{2}$-solution.   
In addition, if $f \in W^{1,2}(0,T;L^{2}(U))$, then the unique solution $u$ satisfies the following properties\/\textup{:} 
\begin{gather*}
  k_{\alpha} * \bigl\|\partial_{t}\bigl[k_{1-\alpha}*(u-u_{0})\bigr]\bigr\|_{L^{2}(U)}^{\,2}
  \in L^{\infty}(0,T),
  \\   
  \esssup_{t \in (0,T)} \|u(t,\bullet)\|_{W_{0}^{1,p}(\Omega_{t})} < \infty. 
\end{gather*} 
\end{theorem}

\section*{Acknowledgements} 
This work is a part of the author's Ph.D.~thesis.
The author would like to express sincere gratitude to his supervisor, Professor  Goro~Akagi 
for numerous discussions, which provided valuable insights.
Furthermore, the author is deeply grateful for the careful reading of earlier drafts
and for many valuable comments and suggestions. 
%The author is supported by JST SPRING, Grant Number JPMJSP2114.
%In addition, 
This work was supported by the Research Institute for Mathematical Sciences, an International Joint Usage/Research Center located in Kyoto University.
The author is also grateful for the hospitality of the Erwin Schr\"odinger International Institute for Mathematics and Physics, where the author had an occasion to present this work during the thematic program ``Free Boundary Problems''.

\section*{Statements and Declarations} 

\subsection*{Funding}
This work was supported by JST SPRING, Grant Number JPMJSP2114.

\subsection*{Competing Interests}
The author declares that there are no competing interests.

\subsection*{Data Availability}
This study does not involve any datasets.

%%%%%%%%%%%%%%%%%%%%%%%%%
%%%%%%%%%%%%%%%%%%%%%%%%%
%%%%%%%%%%%%%%%%%%%%%%%%%
%%%%%%%%%%%%%%%%%%%%%%%%%
%%%%%%%%%%%%%%%%%%%%%%%%%
%%%%%%%%%%%%%%%%%%%%%%%%%

%\bibliographystyle{plain}
%\bibliography{bibliography}

\end{document}